\documentclass{amsart}
\usepackage{amssymb,latexsym}
\usepackage{amsmath,amsfonts,amsthm}
\usepackage[dvips]{graphicx}
\usepackage[all]{xy}
\newtheorem{theorem}{Theorem}[section]
\newtheorem{prop}[theorem]{Proposition}
\newtheorem{lemma}[theorem]{Lemma}
\newtheorem{remark}[theorem]{Remark}

\newtheorem{definition}[theorem]{Definition}
\newtheorem{cor}[theorem]{Corollary}

\def\co{\colon\thinspace}

\def\ep{\epsilon}

\begin{document}
\title[Boundary depth, Hamiltonian dynamics, and coisotropic submanifolds]{Boundary depth in  Floer theory and its applications to Hamiltonian dynamics and coisotropic submanifolds}\author{Michael Usher}\address{Department of Mathematics, University of Georgia, Athens, GA 30602}\email{usher@math.uga.edu}
\begin{abstract}  We assign to each nondegenerate Hamiltonian on a closed symplectic manifold a Floer-theoretic quantity called its ``boundary depth,'' and establish basic results about how the boundary depths of different Hamiltonians are related.  As applications, we prove that certain Hamiltonian symplectomorphisms supported in displaceable subsets have infinitely many nontrivial geometrically distinct periodic points, and we also significantly expand the class of coisotropic submanifolds which are known to have positive displacement energy.  For instance, any coisotropic submanifold of contact type (in the sense of Bolle) in any closed symplectic manifold  has positive displacement energy, as does any stable coisotropic submanifold of a Stein manifold.  We also show that any stable coisotropic submanifold admits a Riemannian metric that makes its characteristic foliation totally geodesic, and that this latter, weaker, condition is enough to imply positive displacement energy under certain topological hypotheses.
\end{abstract}
\maketitle

\section{introduction}
A nondegenerate Hamiltonian $H$ on a closed symplectic manifold $(M,\omega)$ has an associated Floer chain complex $CF_*(H,J)$ (see, \emph{e.g.}, \cite{Sal} for a survey; we use $J$ here to denote a suitable family of almost complex structures, together with an abstract perturbation in the sense of \cite{FO},\cite{LT} in the case where such perturbations are needed), whose homology is equal to the quantum cohomology $QH^*(M)=H^*(M)\otimes\Lambda$, independently of $H$.  This fact, which by the end of the 1990s had been proven for arbitrary $(M,\omega)$ (\cite{FO},\cite{LT}), has as an immediate corollary the remarkable fact that the number of fixed points of the time-$1$ map of the Hamiltonian flow of $H$ is at least equal to the sum of the Betti numbers of $M$, thus establishing a variant of a famous conjecture of Arnol'd.  

The fact that $CF_*(H,J)$ has homology (and in fact chain homotopy type) independent of $H$, while allowing one to prove the aforementioned important result about \emph{all} Hamiltonians, might seem to make Floer homology ill-suited to detecting properties that only obtain for some Hamiltonians.  However, this has proven not to be the case, largely as a result of the fact that the chain complex $CF_*(H,J)$ admits a natural \emph{filtration} by $\mathbb{R}$; thus for each $\lambda\in\mathbb{R}$ one has a chain complex $CF_{*}^{\lambda}(H,J)$.  The isomorphism type of $CF_{*}(H,J)$ as a $\mathbb{R}$-filtered chain complex is independent of the auxiliary data $J$ (as follows from the proof of Corollary 4.5 of \cite{Ohsurvey}, for instance; see Lemma \ref{htopycopy} below for a more general statement), but the filtration-dependent properties of $CF_*(H,J)$ (in particular the filtered homologies $HF_{*}^{\lambda}(H)$) are rather sensitive to $H$.  This suggests a program of studying individual Hamiltonians on symplectic manifolds by extracting invariants from the filtered chain isomorphism types of their Floer chain complexes.

The invariants that have so far most productively been used in this spirit are the so-called \emph{spectral numbers}, which were introduced in the symplectically aspherical case in \cite{Schwarz} and were extended to general closed symplectic manifolds by Oh (see \cite{Oh1}).  Here, recalling that the Floer homology $HF_*(H)$ is \emph{canonically} isomorphic to the quantum cohomology $QH^*(M)$, one associates to a class $a\in QH^*(M)\setminus\{0\}$ the spectral number $\rho(H;a)$, given by the infimum of the real numbers $\lambda$ with the property that the class $a$ is represented in the filtered complex $CF_{*}^{\lambda}(H,J)$.  (We give a few more details in Section \ref{spectralsub} below and refer to \cite{Ohsurvey} for a full survey.)  By now, there have been several applications of the spectral numbers; for a sampling of these, we refer the reader to \cite{EP}, \cite{G06}, \cite{Oh3},  \cite{U3}.

In this paper, we introduce a new filtration-based invariant of a nondegenerate Hamiltonian $H$, called its \emph{boundary depth} and denoted $\beta(H,J)$ (in fact, $\beta(H,J)$ is independent of $J$, and so is later denoted $\beta(H)$).  Where the boundary operator on the Floer chain complex is denoted by $\partial_{H,J}$ (for appropriate auxiliary data $J$), $\beta(H,J)$ is defined as the infimum of all nonnegative numbers $\beta$ with the property that, for any $\lambda\in\mathbb{R}$, one has \[ CF_{*}^{\lambda}(H,J)\cap\partial_{H,J}\left(CF_{*}(H,J)\right)\subset \partial_{H,J}\left(CF_{*}^{\lambda+\beta}(H,J)\right).\]  Thus, for $\beta>\beta(H,J)$, any chain $c$ in the Floer complex that is a boundary is in fact the boundary of a chain with filtration level at most $\beta$ larger than the filtration level of $c$.

Since there typically exist chains in the Floer complex with arbitrary high filtration level, it is not immediately obvious that $\beta(H,J)$ is finite, \emph{i.e.}, that there exist any numbers $\beta$ with the property that the last sentence of the previous paragraph holds.  This does, however, turn out to be the case.  A similar phenomenon to this finiteness was proven in the context of Lagrangian Floer homology as Proposition A.4.9 of \cite{FOOO}, but the first proof applicable to Hamiltonian Floer homology seems to have been the proof of the last sentence of Theorem 1.3 of \cite{U}, which in fact applies in a rather general algebraic setting that includes Hamiltonian Floer homology as a special case.   Almost immediately after the preliminary version of \cite{U} was completed, Oh observed that, in the case of Hamiltonian Floer homology, the finiteness can be deduced more geometrically, and that this approach produces what can be translated into an effective bound on $\beta(H,J)$ (this is Proposition 8.8 of the current revision of \cite{OhCerf}).  This bound, and the approach that leads to it, will be quite important for our purposes below.      

We will state our main results about the behavior of $\beta(H)$ presently.  In the theorem below, we consider smooth periodic Hamiltonians $H\co (\mathbb{R}/\mathbb{Z})\times M\to\mathbb{R}$, which induce time-dependent Hamiltonian vector fields $X_{H}$ defined by $\iota_{X_H}\omega=d(H(t,\cdot))$ and flows $\{\phi_{H}^{t}\}_{0\leq t\leq 1}$ defined by $\frac{d\phi_{H}^{t}}{dt}=X_{H}(t,\phi_{H}^{t}(\cdot)).$  We will typically assume that $H$ is nondegenerate in the sense that, at each fixed point $p$ of $\phi_{H}^{1}$, the linearization $(\phi_{H}^{1})_*$ does not have $1$ as an eigenvalue.   Also, we set \[ \|H\|=\int_{0}^{1}\left(\max_MH(t,\cdot)-\min_M H(t,\cdot)\right)dt.\] 

\begin{theorem}\label{mainbeta} The boundary depth $\beta$ satisfies the following properties: \item[(i)] $\beta(H,J)$ is independent of $J$, and hence will be denoted $\beta(H)$.
\item[(ii)] If $H$ and $K$ are two nondegenerate Hamiltonians we have \[ |\beta(H)-\beta(K)|\leq \|H-K\|.\]
\item[(iii)(Proposition 8.8, \cite{OhCerf})] $\beta(H)\leq\|H\|$.
\item[(iv)] If $H$ and $K$ are two nondegenerate Hamiltonians with either $K\leq 0$ everywhere or $K\geq 0$ everywhere such that, for some compact set $S\subset M$, we have \[ \phi_{H}^{1}(S)\cap S=\varnothing\mbox{ and }|K(t,m)|\leq\ep\mbox{ for all }t\in\mathbb{R}/\mathbb{Z},m\in M\setminus S,\] then \[ \beta(K)\leq 2\|H\|+2\ep.\] 
\item[(v)] If $H$ and $K$ are two nondegenerate Hamiltonians such that $\phi_{H}^{1}=\phi_{K}^{1}$ and the paths $\{\phi_{H}^{t}\}_{0\leq t\leq 1}$ and $\{\phi_{K}^{t}\}_{0\leq t\leq 1}$ are homotopic rel endpoints in the Hamiltonian diffeomorphism group, then \[ \beta(H)=\beta(K).\]
\end{theorem}

\begin{proof} (i) follows from (ii), which is proven as Proposition \ref{Lip}.

As noted, (iii) is Proposition 8.8 of \cite{OhCerf}, which for convenience we reprove as Corollary \ref{hoferdepth} below.

(iv) is Corollary \ref{dispbeta}, which follows quickly from (v) and Lemma \ref{disp1}.

(v) is Corollary \ref{depthhtopyinvt}. 
\end{proof}

Theorem \ref{mainbeta}(v) above follows immediately from the following more general result, proven as Lemma \ref{htopyinvt} below, which may be of independent interest.  Recall that a Hamiltonian $H\co (\mathbb{R}/\mathbb{Z})\times M\to\mathbb{R}$ is called \emph{normalized} if, for all $t$, we have $\int_{M}H(t,\cdot)\omega^n=0$.

\begin{lemma}\label{htopycopy}   Suppose that $H_0$ and $H_1$ are two normalized, nondegenerate Hamiltonians such that $\phi_{H_0}^{1}=\phi_{H_1}^{1}$ and the paths $t\mapsto \phi_{H_i}^{t}$ are homotopic rel endpoints in the Hamiltonian diffeomorphism group.  Then for  $J_i\in\mathcal{J}^{reg}(H_i)$, there is an isomorphism of chain complexes \[ \Phi\co CF_*(H_0, J_0)\to CF_*(H_1,J_1)\] which, for each $\lambda\in\mathbb{R}$, restricts to an isomorphism of the filtered Floer chain complexes $CF_{*}^{\lambda}(H_i,J_i)$.\end{lemma}

In particular this implies that \emph{any} invariant of the filtered chain isomorphism type of $CF_*(H,J)$ will in fact be an invariant of the homotopy class of the path $t\mapsto \phi_{H}^{t}$.  Examples include both the spectral number $\rho(H;a)$\footnote{Strictly speaking, to fit the spectral number into this framework one needs to augment Lemma \ref{htopycopy} with a statement (which is readily seen to be true, though we omit the proof) to the effect that the map $\Phi$ is suitably compatible with the isomorphism between Floer homology and quantum cohomology}  and the boundary depth $\beta(H)$. 

The proofs of the various parts of Theorem \ref{mainbeta} are consequences of the basic structure of the filtered Floer chain complexes $CF_{*}^{\lambda}(H,J)$ and of the various maps connecting these complexes as the Hamiltonian $H$ varies.  As is well-known, if $H$ and $K$ are two nondegenerate Hamiltonians and $J,J'$ are suitable  auxiliary data, the chain complexes $CF_*(H,J)$ and $CF_*(K,J')$ are related by maps $\Phi_{HK}\co CF_*(H,J)\to CF_*(K,J')$ and $\Phi_{KH}\co CF_*(K,J')\to CF_*(H,J)$ such that $\Phi_{KH}\circ \Phi_{HK}$ and $\Phi_{HK}\circ\Phi_{KH}$ are chain homotopic to the identity.  The effects of the maps $\Phi_{HK}$ and $\Phi_{KH}$ on the filtrations of the respective Floer groups satisfy some basic properties that we shall recall below, and these properties  are standard ingredients in the proofs of the properties of the spectral numbers $\rho(\cdot,a)$ (see, \emph{e.g.}, \cite{Ohsurvey}).  A fact which has been used somewhat less is that the \emph{chain homotopies} $\mathcal{K}\co CF_*(H,J)\to CF_*(H,J)$ and $\mathcal{K}'\co CF_*(K,J')\to CF_*(K,J')$ that link  $\Phi_{KH}\circ \Phi_{HK}$ and $\Phi_{HK}\circ\Phi_{KH}$ to the identity also have predictable effects on the filtrations of, respectively, $CF_*(H,J)$ and $CF_*(K,J')$, and it is this fact which underlies most of Theorem \ref{mainbeta}.  This point was already exploited in Oh's proof of Theorem \ref{mainbeta}(iii) (as Proposition 8.8 of \cite{OhCerf}).

Lemma \ref{htopyinvt} (and hence Theorem \ref{mainbeta}(v)) arise by a similar analysis, but now combined with an approach used in \cite{K} that uses a modified version of the equation for a Floer connecting orbit in order to obtain a map between Floer chain complexes which obeys sharper bounds than usual on its effects on the filtrations. 

As a byproduct of the analysis that leads to Theorem \ref{mainbeta}, we also prove the following result, which generalizes 
(and in fact slightly sharpens) a result proven as Theorem 2.1 in \cite{CR} for the  case in which $\omega$ vanishes on $\pi_2(M)$.  I am grateful to O. Cornea for asking me whether this result could be generalized.

\begin{prop}\label{retractcopy}Let $H_0$ be a nondegenerate Hamiltonian on any closed symplectic manifold $(M,\omega)$, and let $J_0\in \mathcal{J}^{reg}(H_0)$.  Then there is $\delta>0$ with the following property.  If $H_1$ is any nondegenerate Hamiltonian on $M$ with $\|H_1-H_0\|<\delta$, and if $J_1\in \mathcal{J}^{reg}(H_1)$, then the chain complex $CF_*(H_0,J_0)$ is a retract of the chain complex $CF_*(H_1,J_1)$ (\emph{i.e.}, there are chain maps $i\co CF_*(H_0,J_0)\to CF_*(H_1,J_1)$ and $r\co CF_*(H_1,J_1)\to CF_*(H_0,J_0)$ such that $r\circ i$ is the identity).
\end{prop}

\begin{proof}
This is proven below as Corollary \ref{retract}.
\end{proof}

In particular, this shows that any nondegenerate Hamiltonian which is sufficiently Hofer-close to $H_0$ has at least as many periodic orbits as does $H_0$; naively, one might expect that the nearby Hamiltonian would need to be $C^1$-close to $H_0$ (as is necessary in order to guarantee that the relevant time-one maps are $C^0$-close to each other) in order for this to occur.

\subsection{Infinitely many periodic points for certain Hamiltonian symplectomorphisms}

 Various results in the literature (for instance, Theorem 5.5 of \cite{Schwarz} and Corollary 10.2 of \cite{FS})  assert that, under fairly strong conditions on the ambient manifold $(M,\omega)$ (in particular, $[\omega]$ is typically required to vanish on $\pi_2(M)$), a symplectomorphism induced by a Hamiltonian which is supported in a displaceable subset of $M$ must have infinitely many geometrically distinct, nontrivial periodic points; recall that a subset $W\subset M$ is called \emph{displaceable} if there is a Hamiltonian $H\co (\mathbb{R}/\mathbb{Z})\times M\to\mathbb{R}$ such that $\phi_{H}^{1}(\bar{W})\cap \bar{W}=\varnothing$.  Here, making use of Theorem \ref{mainbeta}(iv), we attempt to ``localize'' these results; specifically, we assume only that $[\omega]$ vanishes on the second homotopy group of some open subset $V\subset M$, and we consider Hamiltonians supported in some displaceable subset of $V$ (actually, to formulate the result, we choose an open set $U$ with $\bar{U}\subset V$ and assume the support of the Hamiltonian is contained in $U$).  A few tradeoffs are required in order to get analagous results. First, we in fact need to assume that the Hamiltonian has support $W$ which is not merely displaceable but has \emph{displacement energy} $e(W,M)$ less than a constant $c_{UV}$ depending on $U$ and $V$; here \[ e(W,M)=\inf\left\{\|H\|\left|\begin{array}{c} H\co \mathbb{R}/\mathbb{Z}\times M\to \mathbb{R}, H\mbox{ is compactly supported}\\ \mbox{ and }\phi_{H}^{1}(\bar{W})\cap \bar{W}=\varnothing\end{array}\right.\right\}.\]  (In fact, in light of Remark \ref{VM} $c_{UV}$ can be taken to be $\infty$ if $\omega$ vanishes on $\pi_2(M)$ and $V=M$, analagously to the results in \cite{Schwarz},\cite{FS}).   Second, and more technically, we need to impose a hypothesis on the Hamiltonian, namely that the Hamiltonian is nonpositive and ``has a flat autonomous minimum'' (see  Definition \ref{flat}).  The following is proven below as Theorem \ref{inf}.
 
\begin{theorem}\label{infcopy}  Let $U$ and $V$ be open subsets of the closed symplectic manifold $(M,\omega)$ with smooth (possibly empty) orientable  boundaries, and $\bar{U}\subsetneq V$.  Assume that $[\omega]$ vanishes on $\pi_2(\bar{V})$.  Then the constant $c_{UV}>0$ of Proposition \ref{energypartial} obeys the following properties.  Let $K\co (\mathbb{R}/\mathbb{Z})\times M\to (-\infty,0]$ be a not-identically-zero Hamiltonian with a flat autonomous minimum at a point $p\in U$.  Assume that $K^{-1}((-\infty,0))\subset (\mathbb{R}/\mathbb{Z})\times W$, where $W\subset U$ has displacement energy $e(W,M)$ satisfying $e(W,M)<c_{UV}$.  Then the symplectomorphism $\phi_{K}^{1}$ has infinitely many geometrically distinct, nontrivial periodic points.\end{theorem}

To clarify what is being proven here, if $m$ is a periodic point of $\phi_{K}^{1}$ (say with period $k$), define the set \[ \mathfrak{o}_m=\{\phi_{K}^{t}(m)|t\in [0,k]\} \] (note that this definition gives the same set regardless of whether $k$ is taken to be the minimal period of $m$ or some other period). The periodic point $m$ is called \emph{nontrivial} provided that $\mathfrak{o}_m\neq \{m\}$.  Two periodic points $m,m'$ are called \emph{geometrically distinct} if $\mathfrak{o}_{m}\neq \mathfrak{o}_{m'}$.  In particular, if $K$ is autonomous and if $\gamma$ is a nonconstant periodic orbit of $X_K$, the various points on $\gamma$ are \emph{not} considered to be geometrically distinct from one another as periodic points.

\begin{cor}\label{zerodisp} If $S\subset M$ is a compact set of the closed symplectic manifold $(M,\omega)$ with $e(S,M)=0$, and if $S$ has a neighborhood $V\subset M$ such that $\langle [\omega],\pi_2(\bar{V})\rangle=0$, then there is a neighborhood $W$ of $S$ with the property that, for  any nonpositive and not-identically-zero Hamiltonian $K$ with support in $(\mathbb{R}/\mathbb{Z})\times W$ which has a flat autonomous minimum, the map $\phi_{K}^{1}$ has infinitely many nontrivial geometrically distinct periodic points.
\end{cor}

\begin{proof} After possibly replacing $V$ by an open subset of itself that has smooth boundary, choose an open neighborhood $U$ of $S$ with smooth boundary such that $S\subset \bar{U}\subsetneq V$.  This determines a constant $c_{UV}>0$.  Since $e(S,M)=0$ there is a Hamiltonian $H\co(\mathbb{R}/\mathbb{Z})\times M\to \mathbb{R}$ with $\phi_{H}^{1}(S)\cap S=\varnothing$ and $\|H\|<c_{UV}$.  By continuity, there will be an open neighborhood $W$ of $S$ such that $\phi_{H}^{1}(\bar{W})\cap \bar{W}=\varnothing$, and we may apply Theorem \ref{infcopy} to this set $W$.
\end{proof}

Examples of compact sets $S\subset M$ with $e(S,M)=0$ include, by \cite{Gurel}, any closed  submanifold $S$ whose normal bundle has a nonvanishing section and which is ``totally non-coisotropic'' (\emph{i.e.}, for each $p\in S$ it should be true that the symplectic orthogonal complement $(T_pS)^{\perp_{\omega}}$ is \emph{not} contained in $T_pS$; obvious examples include any $S$ which is a symplectic submanifold or which has less than half the dimension of $M$), and, by \cite{LS},\cite{Po}, also any closed submanifold $S$ of half the dimension of $M$ whose normal bundle has a nonvanishing section and which is \emph{not} Lagrangian.  In fact, by applying a stabilization trick of the sort used in \cite{Schlenk} (replace $S\subset M$ with $S\times S^1\subset M\times T^* S^1$), one can still obtain the conclusion of Corollary \ref{zerodisp} for such submanifolds $S$ even when their normal bundles do not admit nonvanishing sections; we leave the details of this to the reader.

\subsection{Coisotropic submanifolds}
Recall that a submanifold $N$ of the symplectic manifold $(M,\omega)$ is called \emph{coisotropic}  if one has, at all $p\in N$, $(T_pN)^{\perp_{\omega}}\subset T_pN$, where of course $(T_pN)^{\perp_{\omega}}$ denotes the subspace of $T_pM$ which is the orthogonal complement of $T_pN$ with respect to the nondegenerate bilinear form $\omega$.   In recent years, some progress has been made in establishing various sorts of rigidity results for such submanifolds.  

Obviously, any Lagrangian submanifold is coisotropic, and if $N$ is a Lagrangian submanifold of a tame symplectic manifold $(M,\omega)$ then a famous result of Chekanov \cite{Ch} shows that the displacement energy $e(N,M)$ is positive.   In fact, Chekanov gives an effective lower bound for $e(N,M)$, namely the minimal area of a nonconstant pseudoholomorphic sphere in $M$ or disc in $M$ with boundary on $N$,  and shows that as long as $\|H\|$ is below this threshold and $H$ is nondegenerate then  $\phi_{H}^{1}(N)\cap N$ will have at least $\sum_{i=0}^{\dim N}b_i(N)$ points.  Ideally, one would like to generalize Chekanov's result to the case where $N$ is a coisotropic submanifold.

In fact, one can in principle hope for a stronger result than the assertion that $N$ has positive displacement energy: note that in a coisotropic submanifold $N$ the distribution $(TN)^{\perp_{\omega}}$ is integrable (this well-known fact is an easy exercise based on the fact that $\omega$ is closed), and so is tangent to a foliation $\mathcal{F}$ on $N$, called the \emph{characteristic foliation}.  A stronger conjecture would assert that, if $\|H\|$ is sufficiently small, $\phi_{H}^{1}$ will necessarily have (possibly a certain number of) ``leafwise intersections'' on $N$, \emph{i.e.}, points $p\in N$ such that $\phi_{H}^{1}(p)$ lies on the same leaf as $p$.  In the Lagrangian case, the only leaves of $\mathcal{F}$ are the connected components of $N$, so this specializes to Chekanov's result, while in the opposite extreme case that $M$ is closed and $M=N$, so that the leaves are singletons, this is a consequence of the proof of the Arnol'd conjecture. In intermediate dimensions, the first result in this direction dates back to Moser \cite{Mo} (under a hypothesis of $C^1$-smallness, rather than Hofer-smallness), and more recent results include those in \cite{Drag}, \cite{AF}, and \cite{Zi}, as well as Theorem 2.7 (iii) of \cite{G06}.  Each of these results requires rather strong hypotheses on one or both of $N$ and $M$  (\emph{e.g.}, $N$ should have restricted contact type with $M$ a subcritical Stein manifold, or the characteristic foliation of $N$ should be a smooth fiber bundle), and the general expectation seems to be that an arbitrary coisotropic submanifold $N$ cannot be expected to satisfy such a leafwise intersection property.

Here, we consider the weaker question of showing that a coisotropic submanifold $N$ has positive displacement energy $e(N,M)$.  Our results generalize results due to \cite{G06} and \cite{K2}, in each of which $N$ is assumed to have a property called \emph{stability} which dates back (under a different name) to \cite{B} and whose definition we shall recall in Section \ref{geodco}.  We prove that:

\begin{theorem} \label{stablemaincopy} Let $(M,\omega)$ be a closed symplectic manifold and let $N$ be a stable coisotropic submanifold of $N$ with the property that $G_N:=\{\int_{S^2}u^*\omega|u\co S^2\to N\}$ is a discrete subgroup of $\mathbb{R}$.  Then there is a constant $c$, depending only on a tubular neighborhood of $N$ in $M$, such that $e(N,M)\geq c$.
\end{theorem}   
\begin{proof} This is proven below as Theorem \ref{stablemain}.\end{proof}

Note that the discreteness hypothesis here is only on $N$, not on $M$; we are thus assuming much less about the symplectic topology of $M$ than has been assumed in any similar previous results about stable coisotropic submanifolds.  In \cite{G06}, for instance, Ginzburg proved this result for the special case that $(M,\omega)$ is symplectically aspherical (\emph{i.e.}, $c_1(M)$ and $[\omega]$ evaluate trivially on $\pi_2(M)$).  Ginzburg's result was  generalized by Kerman in \cite{K2} to the case that $c_1(M)$ and  $[\omega]$ evaluate proportionally to each other on $\pi_2(M)$, with $\omega$ still required to be spherically rational (on all of $M$).  

Of course, the discreteness property for $G_N$ that we require is satisfied automatically when, for instance, $N$ has contact type in the sense of \cite{B}, since then $\omega|_N$ is exact.  Thus we have proven that any contact-type coisotropic submanifold of a closed symplectic manifold has positive displacement energy.  Also, by using the fact that Theorem \ref{stablemaincopy} includes a lower bound on $e(N,M)$ depending only on a neighborhood of the submanifold, together with a result from \cite{LiMa} about embeddings of compact subsets of Stein manifolds, it follows that:

\begin{cor} Any closed stable coisotropic submanifold of a Stein manifold has positive displacement energy.\end{cor}
\begin{proof}See Corollary \ref{stein}.\end{proof}

It follows from our proof that the constant $c$ of Theorem \ref{stablemaincopy} depends only on the following three quantities: (i) the positive generator of the group $G_N$ (considered to be $\infty$ if this group is trivial);   (ii) for any chosen almost complex structure $J_0$ on $M$,  the minimal energy that a $J_0$-holomorphic curve must have in order to intersect both boundary components of an annular neighborhood of $N$ in $M$; and (iii) the minimal length of a closed curve in a leaf of the foliation  $\mathcal{F}$ which is a geodesic with respect to an appropriate metric on $N$.  In the special case that $\omega$ evaluates discretely on $\pi_2(M)$ and not just on $\pi_2(N)$, a straightforward modification of our proof shows that (i) and (ii) above could be collectively replaced by the positive generator of $G_M:=\{\int_{S^2}u^*\omega|u\co S^2\to M\}$; of course, this constant, even when it is positive, has the disadvantage that it no longer depends only on a neighborhood of $N$.

As noted in Remark 2.4 of \cite{G06} and discussed in more detail in section \ref{geodco}, stability is a very restrictive condition to impose on $N$.  For instance, when $N$ is Lagrangian, $N$ is stable if and only if it is a torus, and more generally all closed leaves of the characteristic foliation on a stable coisotropic submanifold must be tori.  Accordingly we make some effort to establish results under a more modest assumption on $N$.\footnote{Shortly after the initial posting of this article, I learned of similar work along these lines that had been independently carried out by B. Tonnelier \cite{Ton}.}  The assumption that we impose is that $N$ should admit a metric $h$ with respect to which the leaves of the characteristic foliation $\mathcal{F}$ are totally geodesic (in other words, every geodesic initially tangent to a leaf remains in the leaf).  In terms of the Levi-Civita connection $\nabla$ associated to $h$, this amounts to the requirement that, if $X,Y$ are local sections of $TN^{\perp_{\omega}}$, then $\nabla_X Y$ is also a local section of $TN^{\perp_{\omega}}$ (in fact, one easily sees that it is enough to verify this condition for the special case that $X=Y$).  Of course, this is still a fairly strong condition---most foliations are not totally geodesic for any metric---but it at least is true that every Lagrangian submanifold satisfies this condition, and that the condition does not \emph{a priori} rule out any particular manifolds from appearing as closed leaves of the foliation.  Further, the condition holds for stable coisotropic submanifolds, as we show:
\begin{prop}\label{sgcopy} If $N$ is a closed coisotropic submanifold of $(M,\omega)$ which is stable, then there is a metric on $N$ with respect to which the characteristic foliation of $N$ is totally geodesic.\end{prop}
\begin{proof} See Proposition \ref{stabimpliesgeod}.\end{proof}

We also show that the converse to Proposition \ref{sgcopy} holds in the special case that $N$ has codimension one; however in higher codimension the converse is rather far from being true.

By the same methods that are used to establish Theorem \ref{stablemaincopy}, together with some elementary arguments about the implications of our condition, we show:
\begin{theorem} \label{geodmaincopy} Suppose that $N$ is a closed coisotropic submanifold of the closed symplectic manifold $(M,\omega)$, that there is a Riemannian metric $h$ on $N$ with respect to which the characteristic foliation of $N$ is totally geodesic, and that $\{\int_{S^2}v^*\omega|v\co S^2\to N\}$ is discrete.  There is $c>0$, depending only on a tubular neighborhood of $N$ in $M$, with the following property.  If $h$ carries no closed geodesics which are contractible in $N$ and are contained in a leaf of the characteristic foliation, then the displacement energy of $N$ is at least $c$.\end{theorem}
\begin{proof}See Theorem \ref{geodmain}.\end{proof}
In fact, if there are no closed geodesics tangent to the foliation on $N$ which are contractible \textbf{in} $\mathbf{M}$, then $N$ is nondisplaceable (\emph{i.e.}, its displacement energy is infinite); this is proven as Theorem \ref{displaceable} below, based on elementary arguments together with the energy-capacity inequality \cite{Schlenk}, \cite{U3}.  

When $N$ is stable, the metric that we construct to prove Proposition \ref{sgcopy} restricts as a flat metric to each of the leaves of the characteristic foliation, and in particular any closed geodesic on a leaf must therefore be noncontractible on the leaf.  Thus Theorem \ref{displaceable} shows that any displaceable, stable coisotropic submanifold of a closed\footnote{or, indeed, Stein, using as usual Theorem 3.2 of \cite{LiMa}} symplectic manifold has a closed curve in a leaf of its characteristic foliation which is not contractible in the leaf.  For such coisotropic submanifolds, this answers the question raised in the second paragraph of \cite{B}.

The proofs of Theorems \ref{stablemaincopy} and \ref{geodmaincopy} follow a  similar strategy to the proof of Theorem 2.7 (i) of \cite{G06}.  Namely, assuming that $\phi_{H}^{1}(N)\cap N=\varnothing$, they consider a suitable Hamiltonian $K\co M\to\mathbb{R}$ supported in a neighborhood $W$ of $N$ which is displaced by $\phi_{H}^{1}$; small nondegenerate perturbations $K^{i}$ of $K$ are seen to admit Floer connecting orbits $u^i$ with certain asymptotics, and a compactness argument gives a Floer-type connecting orbit $u$ for the degenerate Hamiltonian $K$.  The asymptotics of $u$ are such that one can deduce the desired results if one can bound the energy of $u$ in terms of $\|H\|$.   In \cite{G06} this is done by exploiting special properties that are satisfied by the action spectrum of a Hamiltonian on a symplectically aspherical manifold, but when one drops the asphericity condition these properties do not hold.  However, our results on boundary depth---in particular Theorem \ref{mainbeta}(iv)--enable us to bound the energy of $u$ by $2\|H\|$; see Lemma \ref{energypartial}.  Note that Remark 4.3 of \cite{G06} predicted that it might be possible to get an energy bound on this $u$ in terms of $\|K\|$ in the non-aspherical case; one could indeed do this, for instance by using Theorem \ref{mainbeta}(iii) (which is due to Oh \cite{OhCerf}) rather than Theorem \ref{mainbeta}(iv).  However, the fact that one can bound the energy in terms of $\|H\|$ in a non-aspherical context seems not to have been anticipated.

Given this energy bound on $u$, we also use an argument based on a monotonicity formula of Sikorav \cite{Si} in order to argue that (assuming $\|H\|$ to be small enough) the image of $u$ must be contained in a tubular neighborhood of $N$; this is (part of) what enables us to impose some of our topological hypotheses only on $N$ and not on $M$.

We now describe the organization of the paper.  Section \ref{bkgd} summarizes some basic facts about the construction of the Floer complex for a given nondegenerate Hamiltonian and about the maps that relate the complexes associated to different Hamiltonians, with (as we need for later use) special attention paid to the effects that these maps have on the filtrations of the complexes.  These facts should be known to experts; for instance most of them appear in \cite{Ohsurvey}.  We also prove Proposition \ref{retractcopy} at the end of Section \ref{bkgd}. 

 Based on this background, the various parts of Theorem \ref{mainbeta} are proven in Section \ref{sectionbeta}, as is the invariance result Lemma \ref{htopycopy}.  
 
 For our applications, we require both the properties of $\beta$ as given by Theorem \ref{mainbeta} and some known properties of the spectral number $\rho$; these properties are recalled in Section \ref{spectralsub}. 
 
 Our applications, though, depend on some features of certain degenerate Hamiltonians, whereas the Floer homology machinery only directly applies to nondegenerate ones.  Section \ref{nondeg-deg} uses the results of the prior sections combined with some compactness arguments in order to transition us to degenerate Hamiltonians, culminating in an existence result (Proposition \ref{flatlemma}) for solutions to the Floer boundary equation having particular asymptotics and energy bounds, which applies to a fairly broad class of Hamiltonians.
 
 With Proposition \ref{flatlemma} in hand, the proof of Theorem \ref{infcopy} becomes fairly straightforward, and is completed in Section \ref{infsect}.
 
 Section \ref{geodco} discusses some facts, such as Proposition \ref{sgcopy}, about the symplectic and Riemannian geometry of coisotropic submanifolds, and in addition to preparing the ground for our other main results it allows us to prove Theorem \ref{displaceable}, which gives one obstruction to the displaceability of a coisotropic submanifold equipped with a metric making its characteristic foliation totally geodesic. 
 
 Finally, in Section \ref{last} we combine Proposition \ref{flatlemma} with the results of Section \ref{geodco} to prove our main results about coisotropic submanifolds, including Theorems \ref{stablemaincopy} and \ref{geodmaincopy} and some others in a similar spirit.

\section{Background from Floer homology}\label{bkgd}

Assume that $(M,\omega)$ is a closed connected symplectic manifold.  Any smooth function $H\co (\mathbb{R}/\mathbb{Z})\times M\to \mathbb{R}$ then induces a time-dependent Hamiltonian vector field $X_H$ by means of the requirement that, at time $t$, we have \[ \iota_{X_H}\omega=d(H(t,\cdot)).\]  Denote the time-$t$ map of the flow of this vector field by $\phi_{H}^{t}$.  $H$ is called \emph{non-degenerate} if, for each $p\in Fix(\phi_{H}^{1})$, the linearization $(\phi_{H}^{1})_*\co T_pM\to T_pM$ does not have $1$ as an eigenvalue.

If $H$ is nondegenerate, which in particular implies that $\phi_{H}^{1}$ has just finitely many fixed points, we can associate to $H$ and to each $\lambda\in\mathbb{R}$ a Floer homology group $HF^{\lambda}_{*}(H)$.  Let us review the construction and its basic properties, most of the ingredients of which are well-known and date back to \cite{F},\cite{SZ},\cite{HS}.  Let $\widetilde{\mathcal{L}_0M}$ denote the set of equivalence classes $[\gamma,w]$ of pairs $(\gamma,w)$ where $\gamma\co\mathbb{R}/\mathbb{Z}\to M$ is a contractible loop, and where $w\co D^2\to M$ satisfies $w(e^{2\pi it})=\gamma(t)$.  Two pairs $(\gamma,w)$ and $(\gamma',w')$ are deemed equivalent if and only if $\gamma=\gamma'$ and the sphere obtained by gluing $w$ and $w'$ orientation-reversingly along their common boundary $\gamma$ represents a trivial class in the quotient \[ \Gamma:=\frac{\pi_2(M)}{(\ker\langle c_1,\cdot\rangle)\cap (\ker \langle [\omega],\cdot\rangle)}.\]  Now define $\mathcal{A}_H\co \widetilde{\mathcal{L}_0M}\to\mathbb{R}$ by \[ \mathcal{A}_H([\gamma,w])=-\int_{D^2}w^*\omega-\int_{0}^{1}H(t,\gamma(t))dt.\] Where $\mathcal{P}_{H}^{\circ}\subset\widetilde{\mathcal{L}_0M}$ denotes the subset consisting of those pairs $[\gamma,w]$ where $\gamma$ is an integral curve of $X_H$ (so in particular $\phi_{H}^{1}(\gamma(0))=\gamma(0)$), the critical points of $\mathcal{A}_{H}$ are precisely the elements of $\mathcal{P}_{H}^{\circ}$.  Each $[\gamma,w]\in \mathcal{P}_{H}^{\circ}$ has an associated Maslov index $\mu_H([\gamma,w])$, as is explained for instance in \cite{Sal}.  As a group, the chain complex $CF^{\lambda}_{*}(H,J)$ underlying $HF^{\lambda}_{*}(H)$ will be obtained as a downward Novikov completion with respect to $\mathcal{A}_H$ of the span of those $[\gamma,w]\in \mathcal{P}_{H}^{\circ}$ having $\mathcal{A}_H([\gamma,w])\leq \lambda$: specifically, 
\begin{align*} CF^{\lambda}_{*}(H,J)=\left\{\left. \sum_{\begin{array}{ll}[\gamma,w]\in \tilde{\mathcal{P}}_{H}^{\circ}\\ \mathcal{A}_H([\gamma,w])\leq \lambda\end{array}}a_{[\gamma,w]}[\gamma,w]\right| \begin{array}{ll} a_{[\gamma,w]}\in \mathbb{Q},(\forall C\in\mathbb{R})\\(\#\{[\gamma,w]|a_{[\gamma,w]}\neq 0,\mathcal{A}_H([\gamma,w])>C\}<\infty)\end{array} \right\}.\end{align*}  

The aforementioned Maslov index determines a grading on $CF_{*}^{\lambda}(H,J)$, by taking the $k$th graded part $CF_{k}^{\lambda}(H,J)$ equal to the Novikov completion of the span of those $[\gamma,w]$ having $\mu_H=k$ and $\mathcal{A}_H\leq \lambda$.

  The boundary operator $\partial_{H,J}$ of the Floer complex is obtained as follows (see \cite{F},\cite{Sal},\cite{FO},\cite{LT} for details).  For a $(\mathbb{R}/\mathbb{Z})$-parametrized family $\{J^t\}_{0\leq t\leq 1}$ of almost complex structures $J^t$ on $M$ which are compatible with $\omega$, consider the equation
  \begin{equation}\label{bdry}\frac{\partial u}{\partial s}+J^t(u(s,t))\left(\frac{\partial u}{\partial t}-X_H(t,u(s,t))\right)=0\end{equation} for a map $u\co\mathbb{R}\times(\mathbb{R}/\mathbb{Z})\to M$ with finite energy $E(u)=\int_{\mathbb{R}\times\mathbb{R}/\mathbb{Z}}\left|\frac{\partial u}{\partial s}\right|^2dtds$, such that, for certain $[\gamma_{\pm},w_{\pm}]\in\mathcal{P}^{\circ}_{H}$ with $\mu_H([\gamma_-,w_-])-\mu_H([\gamma_+,w_+])=1$, we have $u(s,\cdot)\to \gamma^{\pm}$ as $s\to\pm\infty$ and $[\gamma_+,w_+]=[\gamma_+,w_-\#u]$.  After possibly perturbing the compactified solution space of (\ref{bdry}) by a generic abstract multivalued perturbation $\nu$ of the sort considered in \cite{FO},\cite{LT}, the solution space (modulo $\mathbb{R}$-translation) is represented by a zero-dimensional rational singular chain in $M$ (as shown in, \emph{e.g.}, Theorem 4.2 of \cite{LT}) and we let $n_J([\gamma_-,w_-],[\gamma_+,w_+])$ denote the oriented (rational) number of points in this chain.  Here and elsewhere we use $J$ to denote the pair $(\{J^t\}_{0\leq t\leq 1},\nu)$ where the $J^t$ form a $(\mathbb{R}/\mathbb{Z})$-parametrized family of almost complex structures and $\nu$ is an abstract perturbation.  Let $\mathcal{J}^{reg}(H)$ denote the space of such pairs so that  solution spaces as above are cut out transversally, so that we have a well-defined number  $n_J([\gamma_-,w_-],[\gamma_+,w_+])$ for each $J\in\mathcal{J}^{reg}(H)$.  The reader who does not like abstract perturbations may find it easier to just assume that $M$ is monotone,\footnote{or, at the cost of restricting to Hamiltonians $H$ which belong to a residual subset of the space of nondegenerate Hamiltonians, assume more generally that $M$ is semipositive, as in  \cite{HS}.} so that $\nu$ may be set to zero, and just interpret $J$ as a family of almost complex structures; on the other hand, one does of course need the abstract perturbation machinery of \cite{LT}, \cite{FO} in order to achieve results in the level of generality stated in the introduction.

At any rate, for $J\in\mathcal{J}^{reg}(H)$, we
 set \begin{equation}\label{bdry2} \partial_{H,J}[\gamma_-,w_-]=\sum_{\begin{array}{ll}[\gamma_+,w_+]\in\mathcal{P}_{H}^{\circ}\\ \mu_H([\gamma_+,w_+])=\mu_H([\gamma_-,w_-])-1\end{array}}n_J([\gamma_-,w_-],[\gamma_+,w_+])[\gamma_+,w_+]\end{equation} for each generator $[\gamma_-,w_-]$ of $CF^{\lambda}_{*}(H,J)$.  Note that if there are \emph{no} solutions to (\ref{bdry}) connecting the generator $[\gamma^-,w^-]$ to $[\gamma^+,w^+]$, then the solution space to (\ref{bdry}) is cut out transversally for trivial reasons, and will remain empty for appropriate (small) choices of the abstract perturbation $\nu$. As such, if $n_J([\gamma_-,w_-],[\gamma_+,w_+])\neq 0$, there must exist a solution $u$ to (\ref{bdry}) with $u(s,t)\to \gamma_{\pm}(t)$ as $s\to \pm\infty$ and $[\gamma_+,w_+]=[\gamma_+,w_-\#u]$.  In that case, we have \[ 0\leq  E(u):=\int_{-\infty}^{\infty}\int_{0}^{1}\left|\frac{\partial u}{\partial s}\right|^2 ds dt=\mathcal{A}_H([\gamma_-,w_-])-\mathcal{A}_H([\gamma_+,w_+]).\]  So in this case $\mathcal{A}_H([\gamma_+,w_+])\leq \lambda$   since $\mathcal{A}_H([\gamma_-,w_-])\leq \lambda$.  This, together with a Gromov-Floer compactness argument which shows that the values of $E(u)$ for $u$ solving (\ref{bdry}) cannot accumulate, shows that $\partial_{H,J} [\gamma_-,w_-]$ as defined in (\ref{bdry2}) is an element of $CF_{*}^{\lambda}(H,J)$.  Then extend $\partial_{H,J}$ linearly to all of $CF_{*}^{\lambda}(H,J)$.  As is well-known, one has $(\partial_{H,J})^{2}=0$, and the (filtered) Floer homology $HF^{\lambda}_{*}(H)$ is the homology of this chain complex (which, as the notation suggests, is independent of $J$). 
  
  Obviously $CF^{\lambda}_{*}(H,J)\leq CF^{\mu}_{*}(H,J)$ if $\lambda\leq \mu$, and the boundary operators coincide on the intersection of their domains.  Set $CF_*(H,J)=\cup_{\lambda\in\mathbb{R}} CF^{\lambda}_{*}(H,J)$ and endow it with the obvious boundary operator; the resulting Floer homology $HF_{*}(H)$ is independent of $H$. In fact, $HF_{*}(H)$ coincides as a group with the singular cohomology $H^*(M;\Lambda_{\Gamma,\omega})$, with coefficients in the \emph{Novikov ring} \[\Lambda_{\Gamma,\omega}=\big\{\sum_{g\in \Gamma_{\omega}}b_gg|b_g\in \mathbb{Q},(\forall C\in\mathbb{R})(\#\{g|b_g\neq 0,\int_{S^2}g^*\omega<C\}<\infty)\big\}.\]  Note that $CF_*(H,J)$ has the structure of a chain complex of $\Lambda_{\Gamma,\omega}$-modules, induced by having $g\in \Gamma$ act on a generator $[\gamma,w]$ by gluing a sphere representing $g$ to $w$.  
  
  For the special case where $H(t,m)=\ep f(m)$ where $f\co M\to\mathbb{R}$ is a Morse function and $\ep$ is a sufficiently small positive number, for appropriate $J$ the chain complex $CF_{*}(H,J)$ (and its boundary operator) coincides with the Morse complex $CM_{*}(-\ep f)$, whose boundary operator enumerates negative gradient flowlines of the Morse function $-\ep f$.

Define $\mathcal{L}_H\co CF_*(H)\to \mathbb{R}\cup\{-\infty\}$ by $\mathcal{L}_H(0)=-\infty$ and \[ \mathcal{L}_H\left(\sum c_{[z,w]}[z,w]\right)=\max\{\mathcal{A}_H([z,w])|c_{[z,w]}\neq 0\}.\]  In other words, for $c\in CF_{*}(H)$, \[ \mathcal{L}_H(c)=\inf\{\lambda\in\mathbb{R}|c\in CF_{*}^{\lambda}(H,J)\}.\] Thus the fact that each $CF^{\lambda}$ is preserved by the boundary operator can be expressed as \[ \mathcal{L}_H(\partial_{H,J}c)\leq \mathcal{L}_H(c)\] for all $c\in CF_{*}(H,J)$.  In fact, since the cylinders $u$ involved in the definition of $\partial_{H,J}$ all have strictly positive energy, we have \[ \mathcal{L}_H(\partial_{H,J}c)<\mathcal{L}_H(c).\]

We now review the relationships between the filtrations on the Floer complexes for different Hamiltonians.   Let $H_-,H_+\co (\mathbb{R}/\mathbb{Z})\times M\to\mathbb{R}$ be two nondegenerate Hamiltonians, with $J_{\pm}\in\mathcal{J}^{reg}(H_{\pm})$.  A \emph{homotopy} $\mathcal{H}$ from $(H_-,J_-)$ to $(H_+,J_+)$ is a pair $(\mathbb{H},\mathbb{J})$ consisting of a smooth function $\mathbb{H}\co \mathbb{R}\times(\mathbb{R}/\mathbb{Z})\times M\to\mathbb{R}$ (sending $(s,t,m)$ to $H_s(t,m)$) and 
data $\mathbb{J}$ comprising 
a smoothly $\mathbb{R}\times(\mathbb{R}/\mathbb{Z})$-parametrized family of almost complex structures $J_{s}^{t}$ and an appropriate abstract perturbation $\bar{\nu}$; we require $H_s=H_-$ and $J_{s}^{t}=J_{-}^{t}$ for $s\ll 0$ and $H_s=H_+$ and $J_{s}^{t}=J_{+}^{t}$ for $s\gg 0$, and impose similar restrictions on $\bar{\nu}$.  Let us call the homotopy $\mathcal{H}$ an \emph{interpolating} homotopy if the functions $H_s$ are given by $H_s(t,m)=H_-(t,m)+\beta(s)(H_+(t,m)-H_-(t,m))$, for each $(t,m)\in (\mathbb{R}/\mathbb{Z})\times M$, where $\beta\co \mathbb{R}\to [0,1]$ is a monotone increasing function such that $\beta(s)=0$ for $s\ll 0$ and $\beta(s)=1$ for $s\gg 0$.  

For any Hamiltonian $H$, define \[ \mathcal{E}^+(H)=\int_{0}^{1}\max_{p\in M}H(t,p)dt, \quad \mathcal{E}^-(H)=-\int_{0}^{1}\min_{p\in M}H(t,p)dt,\] and \[ \|H\|=\mathcal{E}^+(H)+\mathcal{E}^-(H)=\int_{0}^{1}\left(\max_{p\in M}H(t,p)-\min_{p\in M}H(t,p)\right)dt.\]

If we choose  $\mathbb{J}$ (which, we recall, is a pair consisting of the family of almost complex structures $J_{s}^{t}$ together with an abstract perturbation $\bar{\nu}$) from a certain residual  subset $\mathcal{J}^{reg}(\mathbb{H})$, 
we can define a map $\Psi_{\mathcal{H}}\co CF_*(H_-,J_-)\to CF_*(H_+,J_+)$ as follows.  For generators $[\gamma_-,w_-]$ of 
$CF_*(H_-,J_-)$ and $[\gamma_+,w_+]$ of $CF_*(H_+,J_+)$ such that $\mu_{H_-}([\gamma_-,w_-])=\mu_{H_-}([\gamma_+,w_+])$, let $m_{\mathbb{J}}([\gamma_-,w_-],[\gamma_+,w_+])$ denote the number of solutions $u\co \mathbb{R}\times (\mathbb{R}/\mathbb{Z})\to M$, counted with appropriate rational weight, to the $\bar{\nu}$-perturbed version of the equation \begin{equation}\label{chainmap} \frac{\partial u}{\partial s}+J_{s}^{t}\left(\frac{\partial u}{\partial t}-X_{H_s}(t,u(s,t))\right)=0\end{equation} which have finite energy $E(u)=\int_{\mathbb{R}\times\mathbb{R}/\mathbb{Z}}|\frac{\partial u}{\partial s}|^2dsdt$ and which satisfy $u(s,t)\to \gamma_{\pm}(t)$ as $s\to \pm\infty$ and $[\gamma_+,w_+]=[\gamma_+,w_-\#u]$.  If $\mu_{H_-}([\gamma_-,w_-])\neq \mu_{H_-}([\gamma_+,w_+])$, set $m_{\mathbb{J}}([\gamma_-,w_-],[\gamma_+,w_+])=0$.  Now, just as with the boundary operator, define \[ \Psi_{\mathcal{H}}([\gamma_-,w_-])=\sum_{[\gamma_+,w_+]\in\mathcal{P}^{\circ}_{H_+}}m_{\mathbb{J}}([\gamma_-,w_-],[\gamma_+,w_+])[\gamma_+,w_+].\]

Let us call $\mathcal{H}$ a regular homotopy if $\mathbb{J}$ belongs to the residual set $\mathcal{J}^{reg}(\mathbb{H})$ of the previous paragraph.

If $\mathbb{J}\in\mathcal{J}^{reg}(\mathbb{H})$, so that $\Psi_{\mathcal{H}}$ is defined, the effect of $\Psi_{\mathcal{H}}$ on the filtrations of the Floer complexes can be understood by means of the following (standard) observation.  If $u$ solves (\ref{chainmap}), one has \begin{equation}\label{hfilt}\mathcal{A}_{H_-}([\gamma,w_-])-\mathcal{A}_{H_+}([\gamma_+,w_+])=E(u)+\int_{-\infty}^{\infty}\int_{0}^{1}\frac{\partial H_s}{\partial s}(t,u(s,t))dtds.\end{equation}  So since $E(u)\geq 0$, a lower bound on the last term provides a lower bound on $\mathcal{A}_{H_-}([\gamma_-,w_-])-\mathcal{A}_{H_+}([\gamma_+,w_+])$.  As with the boundary operator, if this lower bound is not satisfied, then the solution space associated to $([\gamma_-,w_-],[\gamma_+,w_+])$ is empty for the zero abstract perturbation and hence also for all sufficiently small abstract perturbations, and so $m_{\mathbb{J}}([\gamma_-,w_-],[\gamma_+,w_+])$ will be zero.

In particular, suppose that $\mathcal{H}$ is an interpolating homotopy; thus $H_s(t,m)=H_-(t,m)+\beta(s)(H_+(t,m)-H_-(t,m))$, where $\beta\co \mathbb{R}\to [0,1]$ is a monotone increasing function such that $\beta(s)=0$ for $s\ll 0$ and $\beta(s)=1$ for $s\gg 0$.  Then \begin{align*} 
\int_{-\infty}^{\infty}\int_{0}^{1}&\frac{\partial H_s}{\partial s}(t,u(s,t))dtds=\int_{-\infty}^{\infty}\int_{0}^{1}\beta'(s)(H_+-H_-)(t,u(s,t))dtds\\ &\geq \left(\int_{-\infty}^{\infty}\beta'(s)ds\right)\int_{0}^{1}\min_{p\in M}(H_+-H_-)(t,m)dt=-\mathcal{E}^-(H_+-H_-).\end{align*}  Thus (\ref{hfilt}) gives \[ \mathcal{A}_{H_+}([\gamma_+,w_+])\leq \mathcal{A}_{H_-}([\gamma_-,w_-])+\mathcal{E}^-(H_+-H_-)\] whenever $m_{\mathbb{J}}([\gamma_-,w_-],[\gamma_+,w_+])\neq 0$ and $\mathcal{H}$ is an interpolating homotopy.  From the definition of an interpolating homotopy it then immediately follows that:

\begin{prop} \label{interp} \emph{(}\cite{F},\cite{FO},\cite{HS},\cite{LT},\cite{Ohsurvey},\cite{SZ}\emph{)} If $\mathcal{H}$ is a regular interpolating homotopy from $(H_-,J_-)$ to $(H_+,J_+)$, then for each $c\in CF_*(H_-,J_-)$ we have \[ \mathcal{L}_{H_+}(\Psi_{\mathcal{H}}(c))\leq \mathcal{L}_{H_-}(c)+ \mathcal{E}^-(H_+-H_-).\]  Hence, for $\lambda\in\mathbb{R}$, $\Psi_{\mathcal{H}}$ restricts as a map \[ \Psi_{\mathcal{H}}\co CF^{\lambda}_{*}(H_-,J_-)\to CF^{\lambda+\mathcal{E}^-(H_+-H_-)}(H_+,J_+).\]
\end{prop}

A well-known gluing argument shows that $\Psi_{\mathcal{H}}$ is a chain map (regardless of whether or not $\mathcal{H}$ is interpolating), and so it induces maps $\Phi^{H_+}_{H_-}\co HF_*(H_-)\to HF_{*}(H_+)$ on full Floer homology and $\Phi^{H_+}_{H_-}\co HF^{\lambda}(H_-)\to HF^{\lambda+\mu}(H_+)$ on the filtered Floer homology groups, where we can take $\mu=\mathcal{E}^-(H_+-H_-)$ by choosing $\mathcal{H}$ to be interpolating. 

Another well-known gluing argument shows the following.  Given pairs $(H_0,J_0)$, $(H_1,J_1)$, and $(H_2,J_2)$, where the $H_i$ are nondegenerate and $J_i\in\mathcal{J}^{reg}(H_i)$, and given a regular homotopy $\mathcal{H}_0=(\mathbb{H}_0,\mathbb{J}_0)$ from $(H_0,J_0)$ to $(H_1,J_1)$ and a regular homotopy $\mathcal{H}_1=(\mathbb{H}_1,\mathbb{J}_1)$ from $(H_1,J_1)$ to $(H_2,J_2)$, the composition $\Psi_{\mathcal{H}_1}\circ \Psi_{\mathcal{H}_0}$ is equal to $\Psi_{\tilde{\mathcal{H}}_R}$ for $R\gg 0$, where $\tilde{\mathcal{H}}_R=(\tilde{\mathbb{H}}_R,\tilde{\mathbb{J}}_R)$
is defined as follows.  By definition, for $i=0,1$ we have $\mathcal{H}_i=\{(H_{i,s},J_{i,s})\}_{s\in\mathbb{R}}$, where for some $T\gg 0$ $(H_{i,s},J_{i,s})=(H_i,J_i)$ for $s\leq -T$,  and $(H_{i,s},J_{i,s})=(H_{i+1}, J_{i+1})$ for $s\geq T$.  Then put, for $R>T$, $\tilde{\mathcal{H}}_{R}=\{(L_s,\tilde{J}_s)\}_{s\in\mathbb{R}}$ where $(L_s,\tilde{J}_s)=(H_{0,s+2R},J_{0,s+2R})$ for $s<-T$, $(L_s,\tilde{J}_s)=(H_1,J_1)$ for $-T\leq s\leq T$, and $(L_s,\tilde{J}_s)=(H_{1,s-2R},J_{1,s+2R})$ for $s>T$.  

Of interest to us will be the case that $(H_0,J_0)=(H_2,J_2)$.  Then $\tilde{\mathcal{H}}_{R}$ is a homotopy from $(H_0,J_0)$ to itself.  Of course, another homotopy from $(H_0,J_0)$ to itself is the constant homotopy $\mathcal{H}_{cst}$.  At least if the residual set $\mathcal{J}^{reg}(H_0)$ was chosen appropriately, the induced map $\Psi_{\mathcal{H}_{cst}}\co CF_{*}(H_0,J_0)\to CF_{*}(H_0,J_0)$ will be the identity, since this map counts what are in effect index-zero  negative gradient flowlines for the action functional $\mathcal{A}_{H_0}$, and for generic $J_0$ the only such flowlines $u$ are those with $u(s,\cdot)$ constantly equal to a some periodic orbit $\gamma$ of $X_{H_0}$ (moreover, it follows from Lemma 2.4 of \cite{Sal} that the linearized operator at such a constant flowline is bijective by virtue of the nondegeneracy of $H_0$, so no abstract perturbation is needed to cause the solution space to be cut out transversally).

As is well-known, the maps induced on Floer homology by distinct homotopies connecting the same $(H_i,J_i)$ are always chain homotopic; we will be needing more detail about the nature of this chain homotopy (in the particular case of the previous paragraph), so let us recall how the chain homotopy is obtained.  Given $(H_0,J_0)$ and $(H_1,J_1)$, we are considering the following two homotopies from $(H_0,J_0)$ to itself: 
\begin{itemize} \item the constant homotopy $\mathcal{H}_{cst}$, and 
\item A homotopy $\tilde{\mathcal{H}}_R$, obtained by gluing an interpolating homotopy from $(H_0,J_0)$ to $(H_1,J_1)$ to an interpolating homotopy from $(H_1,J_1)$ to $(H_0,J_0)$.\end{itemize}

Note then that $\tilde{\mathcal{H}}_R=\{(L_s,\tilde{J}_s)\}_{s\in\mathbb{R}}$ where \[ L_s(t,m)=H_0+\alpha(s)(H_1(t,m)-H_0(t,m))\] for a certain function $\alpha\co \mathbb{R}\to [0,1]$ which satisfies \begin{itemize} \item $\alpha(s)=0$ for $|s|>R$,
\item $\alpha(0)=1$, \item $\alpha'(s)\geq 0$ for $s<0$, and \item $\alpha'(s)\leq 0$ for $s>0$.  \end{itemize}

To describe the chain homotopy, for $0\leq\tau\leq 1$,  for $s\in \mathbb{R}$, and for $(t,m)\in(\mathbb{R}/\mathbb{Z})\times M$, set \[ K^{\tau}_{s}(t,m)=H_0(t,m)+\tau(L_s(t,m)-H_0(t,m));\] thus \[ K^{\tau}_{s}(t,m)=H_0(t,m)+\tau\alpha(s)(H_1(t,m)-H_0(t,m)).\]

We then define a map $\mathcal{K}\co CF_{*}(H_0,J_0)\to CF_{*}(H_0,J_0)$ by counting, in the familiar way, finite energy solutions to an abstract multivalued perturbation of the equation \[ \frac{\partial u}{\partial s}+\tilde{J}_{s}^{t}\left(\frac{\partial u}{\partial t}-X_{K^{\tau}_{s}}(t,u(s,t))\right),\] where now $\tau$ is allowed to vary freely in $[0,1]$, and  $u$ above contributes to a nonzero matrix element $\langle \mathcal{K}[\gamma_-,w_-],[\gamma_+,w_+]\rangle$ precisely when $u(s,\cdot)\to \gamma_{\pm}$ as $s\to -\infty$, $[\gamma_+,w_+]=[\gamma_+,w_-\#u]$ and $\mu_{H_0}([\gamma_+,w_+])-\mu_{H_0}([\gamma_-,w_-])=-1$ (so the solution would have index $-1$ for fixed $\tau$, and allowing $\tau$ to vary produces a solution space of expected dimension zero).  In such a situation, we have, just as in (\ref{hfilt}), \[ \mathcal{A}_{H_0}([\gamma_-,w_-])-\mathcal{A}_{H_0}([\gamma_+,w_+])=E(u)+\int_{-\infty}^{\infty}\int_{0}^{1}\frac{\partial K_{s}^{\tau}}{\partial s}(t, u(s,t))dtds.\]  Now our above formula for $K_{s}^{\tau}$ gives 

\[
 \int_{-\infty}^{\infty}  \int_{0}^{1}  \frac{\partial K_{s}^{\tau}}{\partial s}(t, u(s,t))dsdt =\tau\int_{-\infty}^{\infty}\int_{0}^{1}\alpha'(s)(H_1(t,u(s,t))-H_0(t,u(s,t))dtds \]\begin{align*}
 =\tau \Big( & \int_{-\infty}^{0} \alpha'(s) \int_{0}^{1}(H_1(t,u(s,t))-H_0(t,u(s,t))dtds\\ &+\left.\int_{0}^{-\infty}\alpha'(s)\int_{0}^{1}(H_1(t,u(s,t))-H_0(t,u(s,t))dtds\right)\end{align*}
 \[
 \geq \int_{0}^{1}\min_{p\in M}(H_1-H_0)(t,p)dt-\int_{0}^{1}\max_{p\in M}(H_1-H_0)(t,p)\]\[=-\mathcal{E}^-(H_1-H_0)-\mathcal{E}^+(H_1-H_0)=-\|H_1-H_0\|.
\]

So since $E(u)\geq 0$, we necessarily have \[ \mathcal{A}_{H_0}([\gamma_+,w_+])\leq \mathcal{A}_{H_0}([\gamma_-,w_-])+\|H_1-H_0\|\] whenever the matrix element $\langle \mathcal{K}[\gamma_-,w_-],[\gamma_+,w_+]\rangle$ is nonzero.  

Well-known arguments (see, \emph{e.g.}, \cite{SZ}) show that $\mathcal{K}$ is indeed a chain homotopy from $\Psi_{\mathcal{H}_{cst}}=I$ to $\Psi_{\tilde{\mathcal{H}}_R}=\Psi_{\mathcal{H}_1}\circ\Psi_{\mathcal{H}_0}$ (here $I$ denotes the identity).  Thus we have: 

\begin{prop}\label{htopylemma}\emph{(}\cite{F},\cite{FO},\cite{HS},\cite{LT},\cite{Ohsurvey},\cite{SZ}\emph{)} If $H_0$ and $H_1$ are nondegenerate Hamiltonians, if $J_i\in\mathcal{J}^{reg}(H_i)$, if $\mathcal{H}_0$ is a regular interpolating homotopy from $(H_0,J_0)$ to $(H_1,J_1)$, and if $\mathcal{H}_1$ is a regular interpolating homotopy from $(H_1,J_1)$ to $(H_0,J_0)$, then there exists a degree-$1$ homomorphism $\mathcal{K}\co CF_*(H_0,J_0)\to CF_*(H_0,J_0)$ with the following properties: \begin{itemize} \item If $c\in CF_*(H_0,J_0)$, then \[ \mathcal{L}_{H_0}(\mathcal{K}c)\leq \mathcal{L}_{H_0}(c)+\|H_1-H_0\|.\] \item Where $\Psi_{\mathcal{H}_i}$ denotes the chain map induced by the homotopy $\mathcal{H}_i$, and $I\co CF_*(H_0,J_0)\to CF_*(H_0,J_0)$ is the identity, we have \[ \Psi_{\mathcal{H}_1}\circ\Psi_{\mathcal{H}_0}-I=\partial_{H_0,J_0}\circ\mathcal{K}+\mathcal{K}\circ\partial_{H_0,J_0}.\]\end{itemize}
\end{prop}

This preparation now enables us to prove Proposition \ref{retractcopy} from the introduction, which generalizes a result from \cite{CR} and which we restate here.

\begin{cor}\label{retract} Let $H_0$ be a nondegenerate Hamiltonian on any closed symplectic manifold $(M,\omega)$, and let $J_0\in \mathcal{J}^{reg}(H_0)$.  Then there is $\delta>0$ with the following property.  If $H_1$ is any nondegenerate Hamiltonian on $M$ with $\|H_1-H_0\|<\delta$, and if $J_1\in \mathcal{J}^{reg}(H_1)$, then the chain complex $CF_*(H_0,J_0)$ is a retract of the chain complex $CF_*(H_1,J_1)$ (\emph{i.e.}, there are chain maps $i\co CF_*(H_0,J_0)\to CF_*(H_1,J_1)$ and $r\co CF_*(H_1,J_1)\to CF_*(H_0,J_0)$ such that $r\circ i$ is the identity).
\end{cor}

\begin{proof}
All index-one solutions to the Floer boundary equation (\ref{bdry}) (with $H=H_0,J=J_0$) have positive energy, and so it  follows from Gromov-Floer compactness that if $\delta$ is the infimal energy of any such solution then $\delta>0$.  In light of the definition of the Floer boundary operator we hence have, for each $c\in CF_*(H_0,J_0)$, $\mathcal{L}_{H_0}(\partial_{H_0,J_0} c)\leq 
\mathcal{L}_{H_0}(c)-\delta$.  With notation as in Proposition \ref{htopylemma}, define $A\co CF_*(H_0,J_0)\to CF_*(H_0,J_0)$ by $A=\partial_{H_0,J_0}\circ\mathcal{K}+\mathcal{K}\circ\partial_{H_0,J_0}$.  Then for any $c\in CF_*(H_0,J_0)$ we have $\mathcal{L}_{H_0}(Ac)\leq \mathcal{L}_{H_0}(c)-(\delta-\|H_1-H_0\|)<\mathcal{L}_{H_0}(c)$.  But then for any $k\geq 0$ $\mathcal{L}_{H_0}(A^kc)\leq \mathcal{L}_{H_0}(c)-k(\delta-\|H_1-H_0\|)$; since this diverges to $-\infty$ as $k\to\infty$ it follows from the definition of the Floer chain complex that \[ B:=\sum_{k=0}^{\infty}(-A)^k\co CF_*(H_0,J_0)\to CF_*(H_0,J_0) \] is well-defined.  $B$ is a chain map (of complexes over the Novikov ring) since $A$ is, and of course we have $B\circ(I+A)=I$.  But Proposition \ref{htopylemma} says that $\Psi_{\mathcal{H}_1}\circ \Psi_{\mathcal{H}_0}=I+A$, so setting $i=\Psi_{\mathcal{H}_0}$ and $r=B\circ \Psi_{\mathcal{H}_1}$ proves the corollary.\end{proof}

\section{The boundary depth $\beta$}\label{sectionbeta}

\begin{definition} If $H\co (\mathbb{R}/\mathbb{Z})\times M\to\mathbb{R}$ is a nondegenerate Hamiltonian and if $J\in \mathcal{J}^{reg}(H)$, we define the boundary depth $\beta(H,J)$ of $(H,J)$ to be the infimum of all numbers $\beta\geq 0$ with the following property: \[ \mbox{ For all $\lambda\in\mathbb{R}$, } \partial_{H,J}(CF(H,J))\cap CF^{\lambda}(H,J)\subset \partial_{H,J}(CF^{\lambda+\beta}(H,J)).\]
\end{definition}

As is mentioned in the introduction, it is not initially obvious that $\beta(H,J)$ is finite; however, it does turn out to be finite, as was shown in \cite{U} and \cite{OhCerf} and as also follows from the results we prove presently.  First, the following result establishes Theorem \ref{mainbeta}(ii). 

\begin{prop}\label{Lip} Whenever $\beta(H_0,J_0)$ and $\beta(H_1,J_1)$ are both defined, we have \[ |\beta(H_0,J_0)-\beta(H_1,J_1)|\leq \|H_1-H_0\|.\]\end{prop}

\begin{proof} By symmetry, it is enough to show that, whenever $\beta>\beta(H_1,J_1)$, we have \[ \beta(H_0,J_0)\leq \beta+\|H_1-H_0\|.\]  So let $\beta>\beta(H_1,J_1)$.

Where $\lambda\in\mathbb{R}$, suppose that $c\in \partial_{H_0,J_0}(CF(H_0,J_0))\cap CF^{\lambda}(H_0,J_0)$.   Use the notation of Proposition \ref{htopylemma}.  Since $\Psi_{\mathcal{H}_0}$ is a chain map, $\Psi_{\mathcal{H}_0}c\in \partial_{H_1,J_1}(CF_*(H_1,J_1))$, and by Proposition \ref{interp} we have $\Psi_{\mathcal{H}_0}c\in CF_{*}^{\lambda+\mathcal{E}^-(H_1-H_0)}(H_1,J_1)$.  Hence by our choice of $\beta$ there is some $b\in CF_{*}^{\lambda+\beta+\mathcal{E}^-(H_1-H_0)}(H_1,J_1)$ such that $\partial_{H_1,J_1} b=\Psi_{\mathcal{H}_0}c$.  Noting that $\mathcal{E}^-(H_0-H_1)=\mathcal{E}^+(H_1-H_0)$, we will then have \[ \mathcal{L}_{H_0}(\Psi_{\mathcal{H}_1}b)\leq \lambda+\beta+\mathcal{E}^-(H_1-H_0)+\mathcal{E}^+(H_1-H_0)=\lambda+\beta+\|H_1-H_0\|.\]  Now (using that $\partial_{H_0,J_0}c=0$ since $\partial_{H_0,J_0}^2=0$) Proposition \ref{htopylemma} shows that \[
c=\Psi_{\mathcal{H}_1}(\Psi_{\mathcal{H}_0}c)-\partial_{H_0,J_0} \mathcal{K}c=\partial_{H_0,J_0}(\Psi_{\mathcal{H}_1}b-\mathcal{K}c).
\]  We've seen that $\Psi_{\mathcal{H}_1}b\in CF_{*}^{\lambda+\beta+\|H_1-H_0\|}(H_0,J_0)$, and Proposition \ref{htopylemma} shows that $\mathcal{K}c\in CF^{\lambda+\|H_1-H_0\|}(H_0,J_0)$ so (since $\beta\geq 0$) we have \[ c\in \partial_{H_0,J_0}\left(CF_{*}^{\lambda+\beta+\|H_1-H_0\|}(H_0,J_0)\right).\]  Since $\beta>\beta(H_1,J_1)$, $\lambda\in\mathbb{R}$, and $c\in CF_{*}^{\lambda}(H_0,J_0)$ were all arbitrary, this shows that \[ \beta(H_0,J_0)\leq\beta(H_1,J_1)+\|H_1-H_0\|,\] and so reversing the roles of $H_0$ and $H_1$ proves the proposition.
\end{proof}
\begin{remark} As a special case of Proposition \ref{Lip} (where $H_0=H_1$), we learn that $\beta(H,J)$ is in fact independent of the choice of $J\in \mathcal{J}^{reg}(H)$, proving Theorem \ref{mainbeta} (i).  Hence from now on we write $\beta(H)$ for $\beta(H,J)$.
\end{remark}

\begin{prop} \label{smallmorse} If $f\co M\to\mathbb{R}$ is a Morse function and $\ep>0$ is sufficiently small, then \[ \beta(\ep f)\leq \|\ep f\|.\]
\end{prop}

\begin{proof}  As noted earlier (and as is well-known), for sufficiently small $\ep$ and for appropriate $J$, the complex $CF_{*}(\ep f,J)$ will coincide with the Morse complex $CM_*(-\ep f)\otimes \Lambda_{\Gamma,\omega}$, with all solutions to the Floer boundary equation (\ref{bdry}) being $t$-independent negative gradient flowlines for $-\ep f$ (that this is true with no condition on $M$ is shown in Lemma 5.1 of \cite{LT}).  Now the energy of such a flowline, say connecting critical points $p_1$ and $p_2$, is just $\ep f(p_2)-\ep f(p_1)$, which is bounded above by the total variation $\|\ep f\|$ of $-\ep f$.  The result then immediately follows from the relevant definitions. 
\end{proof}

We now restate Theorem \ref{mainbeta}(iii).
\begin{cor} \label{hoferdepth}(\cite{OhCerf}, Proposition 8.8) For any nondegenerate Hamiltonian $H$ we have \[ \beta(H)\leq \|H\|.\]
\end{cor}

\begin{proof} Choose a Morse function $f$ and let $\ep>0$.  By the preceding two propositions we have, for small enough $\ep>0$, \[ \beta(H)\leq \beta(\ep f)+\|H-\ep f\|\leq \|\ep f\|+\|H-\ep f\|\leq \|H\|+2\|\ep f\|=\|H\|+2\ep\|f\|.\]  Since this holds for all sufficiently small $\ep>0$ the corollary follows.
\end{proof}

Our goal now will be to analyze the relationship between the boundary depth $\beta$ of (a nondegenerate perturbation of) a Hamiltonian which is supported in some set $S$ to the boundary depth of a Hamiltonian which displaces $S$.  It proves convenient to separately reparametrize time for the two Hamiltonians so that they can be more easily concatenated.  We begin with the following.

\begin{lemma} \label{concat} Let $H\co (\mathbb{R}/\mathbb{Z})\times M\to\mathbb{R}$ be a nondegenerate Hamiltonian, with $H(t,m)=0$ for all $t\in (1/2,1)$ and all $m\in M$.  Let $K\co (\mathbb{R}/\mathbb{Z})\times M\to \mathbb{R}$ be a  Hamiltonian with $supp(K)\subset (1/2,1)\times S$, where $S\subset M$ is a compact subset with the property that \[ \phi_{H}^{1}(S)\cap S=\varnothing.\]  Assume furthermore that the image of $K$ is contained either in $(-\infty,0]$ or in $[0,\infty)$.  Then $H+K$ is nondegenerate, and \[ \beta(H)=\beta(H+K).\]
\end{lemma}

\begin{proof} Since we've assumed that $H(t,\cdot)$ vanishes for $t\in (1/2,1)$ while $K(t,\cdot)$ vanishes for $t\in (0,1/2)$, we see that $\phi_{H+K}^{1}=\phi_{K}^{1}\circ\phi_{H}^{1}$.  From this and the hypothesis on $S$, it's easy to see that $Fix(\phi_{H+K}^{1})=Fix(\phi_{H}^{1})$.  The hypothesis also implies that $Fix(\phi_{H}^{1})\cap S=\varnothing$, so since $M\setminus S$ is open we see that $\phi_{H+K}^{1}$ and $\phi_{H}^{1}$ coincide on a neighborhood of their common fixed point set.  Hence the nondegeneracy of $H$ implies that of $H+K$.

Now we have (noting that $\phi_{H}^{t}=\phi_{H}^{1}$ for $t\in [1/2,1]$, while $\phi_{K}^{t}$ is the identity for $t\in [0,1/2]$) \[ \phi_{H+K}^{t}=\left\{\begin{array}{ll} \phi_{H}^{t} & 0\leq t\leq 1/2 \\ \phi_{K}^{t}\circ \phi_{H}^{1} & 1/2\leq t\leq 1\end{array}\right.\]   Hence if $p\in Fix(\phi_{H+K}^{1})=Fix(\phi_{H}^{1})$, so that in particular $p=\phi_{H}^{1}(p)\notin S$, we see that $\phi_{H+K}^{t}(p)=\phi_{H}^{t}(p)$ for all $t$.  Also, if $p\in Fix(\phi_{H}^{1})$ \[ K(t,\phi_{H}^{t}(p))=0 \] for all $t$, because $\phi_{H}^{t}(p)=p\notin S$ for $0\leq t\leq 1/2$, while $K(t,\cdot)=0$ for $1/2\leq t\leq 1$.

As a result of the above, we have $\mathcal{P}_{H}^{\circ}=\mathcal{P}_{H+K}^{\circ}$ (both sets consist of equivalence classes $[\gamma,w]$ where $\gamma(t)=\phi_{H}^{t}(p)$ and $p\in Fix(\phi_{H}^{1})$).  Furthermore, since $K(t,\gamma(t))$ vanishes identically for any such $[\gamma,w]$, we have \begin{equation}\label{dispequal} \mathcal{A}_{H}([\gamma,w])=\mathcal{A}_{H+K}([\gamma,w])\end{equation} for each $[\gamma,w]\in \mathcal{P}_{H}^{\circ}$.

Now assume that we are in the case that $K\leq 0$ everywhere; at the end of the proof we will indicate how to modify the argument in the case that $K\geq 0$ everywhere. Let $J_-\in\mathcal{J}^{reg}(H+K)$, $J_+\in\mathcal{J}^{reg}(H)$, and let $\mathcal{H}=\{(H_s,J_s)\}_{s\in\mathbb{R}}$ be a regular interpolating homotopy from $(H+K, J_-)$ to $(H,J_+)$.  Thus $H_s(t,m)=(1-\beta(s))K(t,m)+H(t,m)$ where $\beta\co \mathbb{R}\to [0,1]$ has $\beta'(s)\geq 0$, so that $\frac{\partial H_s}{\partial s}\geq 0$ since $K\leq 0$.  Replacing $K$ with $(1-\beta(s))K$ in the above two paragraphs shows that each $H_s$ is nondegenerate, and that $\mathcal{P}_{H_s}^{\circ}=\mathcal{P}_{H}^{\circ}$.  From this it is easy to show (for instance, the proof of Proposition 2.7 of \cite{U2} carries over directly) that the induced chain map $\Psi_{\mathcal{H}}\co CF_{*}(H+K,J_-)\to CF_{*}(H,J_+)$ acts via $\Psi_{\mathcal{H}}[\gamma,w]=[\gamma,w]+A[\gamma,w]$ where $\mathcal{L}_{H}(A[\gamma,w])<\mathcal{A}_{H+K}([\gamma,w])$.  In other words, we have $\Psi_{\mathcal{H}}=I+A$ where $I$ is the identity and $\mathcal{L}_{H}(A(c))<\mathcal{L}_{H+K}(c)$.  A compactness argument regarding the solutions to (\ref{chainmap}) (or, more formally, the facts that there are only finitely many fixed points of $\phi_{H}^{1}$ and that $\Psi_{\mathcal{H}}$ respects the action of the Novikov ring on the chain complexes) in fact shows that for some $\delta>0$ we have $\mathcal{L}_{H}(A(c))\leq \mathcal{L}_{H+K}(c)-\delta$.  

Hence, bearing in mind that the fact that $\mathcal{P}_{H}^{\circ}=\mathcal{P}_{H+K}^{\circ}$ implies that 
$CF_{*}(H+K,J_-)= CF_{*}(H,J_+)$ as $\Lambda_{\Gamma,\omega}$-modules, as in the proof of Corollary \ref{retract}, setting $B=\sum_{k=0}^{\infty}(-A)^{k}$ provides an inverse to $\Psi_{\mathcal{H}}=I+A$. The map  $B\co CF_{*}(H,J_+)\to CF_{*}(H+K,J_-)$ is a chain map simply by virtue of being the inverse of a bijective chain map.  Also, since the functions $\mathcal{L}_H$ and $\mathcal{L}_{H+K}$ coincide by (\ref{dispequal}), the formula for $B$ implies that we have $\mathcal{L}_{H+K}(B(c))= \mathcal{L}_H(c)$ for each $c\in CF_*(H,J_+)$.

If $\beta>\beta(H)$ and $c\in \partial_{H+K,J_-}(CF_{*}(H+K,J_-))\cap CF_{*}^{\lambda}(H+K,J_-)$, then  $\Psi_{\mathcal{H}}(c)\in 
 \partial_{H,J_+}(CF_{*}(H,J))\cap CF_{*}^{\lambda}(H,J_+)$, so there is $b\in CF_{*}^{\lambda+\beta}(H,J_+)$ with $\partial_{H,J_+}b=
 \Psi_{\mathcal{H}}(c)$.  But then $Bb\in CF_{*}^{\lambda+\beta}(H+K,J_-)$ has $\partial_{H+K,J_-}(Bb)=c$.  This proves that $\beta(H+K)\leq\beta$.  $\beta>\beta(H)$ was arbitrary, so $\beta(H+K)\leq\beta(H)$.
 
Likewise if $\beta>\beta(H+K)$ and $c\in \partial_{H,J_+}(CF_{*}(H,J_+))\cap CF_{*}^{\lambda}(H,J_+)$, then since $B$ is a chain map which preseves the filtration level we can find $b\in CF_{*}^{\lambda+\beta}(H+K,J_-)$ such that $\partial_{H+K,J_-}b=Bc$.  Then $\Psi_{\mathcal{H}}b$ will have filtration level at most $\lambda+\beta$ and boundary $c$, proving that $\beta(H)\leq \beta$.  $\beta>\beta(H+K)$ was arbitrary, so $\beta(H)\leq \beta(H+K)$, completing the proof when $K\leq 0$.

When instead $K\geq 0$, replace the homotopy $\mathcal{H}=\{(H_s,J_s)\}_{s\in\mathbb{R}}$ above by one having $H_s(t,m)=\beta(s)K(t,m)+H(t,m)$, where again $\beta'(s)\geq 0$ for all $s$, $\beta(s)=0$ for $s\ll 0$, and $\beta(s)=1$ for $s\gg 0$.  So in this case $\mathcal{H}$ is a homotopy from $(H,J_-)$ to $(H+K,J_+)$, with $\frac{\partial H_s}{\partial s}\geq0$.   Just as before, the induced map $\Psi_{\mathcal{H}}\co CF_*(H,J_-)\to CF_*(H+K,J_+)$ then has form $\Psi_{\mathcal{H}}=I+A$ where $I$ is the identity and $\mathcal{L}_{H+K}(Ac)<\mathcal{L}_H(c)$, and then $\sum_{k=0}^{\infty}(-A)^k$ is an inverse to $\Psi_{H}$ which preserves the filtrations, and the proof may be completed exactly as in the case where $K\leq 0$.
\end{proof}

As a consequence, we obtain:

\begin{lemma} \label{disp1} Let $H\co (\mathbb{R}/\mathbb{Z})\times M\to\mathbb{R}$ be a nondegenerate Hamiltonian, let $S\subset M$ be a compact subset with $\phi_{H}^{1}(S)\cap S=\varnothing$, let $\ep>0$, and let $K\co(\mathbb{R}/\mathbb{Z})\times M\to\mathbb{R}$ be a nondegenerate Hamiltonian such that $K(t,m)=0$ for $0\leq t\leq 1/2$, either $K(t,m)\leq 0$ for all $(t,m)\in(\mathbb{R}/\mathbb{Z})\times M$ or $K(t,m)\geq 0$ for all $(t,m)\in(\mathbb{R}/\mathbb{Z})\times M$, and $|K(t,m)|\leq \ep$ for all $m\in M\setminus S$.  Then \[ \beta(K)\leq 2\|H\|+\ep   .\]
\end{lemma}

\begin{proof} Since $S$ is compact and $\phi_{H}^{1}(S)\cap S=\varnothing$,  we can choose a smooth function $\chi\co M\to [0,1]$ such that $\chi|_S=1$ and $\phi_{H}^{1}(supp\chi)\cap (supp\chi)=\varnothing$.  Set $K'(t,m)=\chi(m)K(t,m)$.
Then $\|K-K'\|\leq \ep$.

Let $\rho\co [0,1]\to [0,1]$ be a smooth monotone increasing function which vanishes to infinite order at $0$ and has $\rho(t)=1$ for $1/2\leq t\leq 1$.  Put $H^{\rho}(t,m)=\rho'(t)H(\rho(t),m)$; this defines a smooth function on $(\mathbb{R}/\mathbb{Z})\times M$ since $\rho'$ vanishes to infinite order at both $0$ and $1$.  Note that $\phi_{H^{\rho}}^{t}=\phi_{H}^{\rho(t)}$.

Now apply Lemma \ref{concat} (with $H$ replaced by $H^{\rho}$ and $K$ replaced by $K'$) to find that $\beta(H^{\rho}+K')=\beta(H^{\rho})$.  But by Corollary \ref{hoferdepth}, \begin{align*} \beta(H^{\rho})&\leq \|H^{\rho}\|=\int_{0}^{1}\left(\max_{p\in M}\rho'(t)H(\rho(t),p)-\min_{p\in M}\rho'(t)H(\rho(t),p)\right)dt
\\& =\int_{0}^{1}\left(\max_{p\in M}H(s,p)-\min_{p\in M}H(s,p)\right)ds=\|H\|.\end{align*}  

Thus by Proposition \ref{Lip}, \[ \beta(K)\leq \beta(H^{\rho}+K')+\|H^{\rho}+(K'-K)\|\leq \beta(H^{\rho}+K')+\|H^{\rho}\|+\|K'-K\| \leq 2\|H\|+\ep.\]

\end{proof}

We now prove a result (stated earlier as Lemma \ref{htopycopy}) about the behavior of the filtered Floer complex under homotopies within the group $Ham(M,\omega)$.  This result is perhaps not surprising in light of various other known results, but seems not yet to be in the literature, and has as an immediate consequence an invariance result both for the boundary depth $\beta$ and the spectral numbers $\rho$ (the invariance of $\rho$ was already known, but this seems to be a  different explanation for the phenomenon).  Our approach is influenced by the methods in \cite{K}.  

\begin{lemma} \label{htopyinvt}  Suppose that $H_0$ and $H_1$ are two normalized, nondegenerate Hamiltonians such that $\phi_{H_0}^{1}=\phi_{H_1}^{1}$ and the paths $t\mapsto \phi_{H_i}^{t}$ are homotopic rel endpoints in $Ham(M,\omega)$.  Then for  $J_i\in\mathcal{J}^{reg}(H_i)$, there is an isomorphism of chain complexes \[ \Phi\co CF_*(H_0, J_0)\to CF_*(H_1,J_1)\] such that, for each $c\in CF_*(H_0,J_0)$, we have \[ \mathcal{L}_{H_0}(c)=\mathcal{L}_{H_1}(\Phi(c)).\] 
\end{lemma}

This immediately yields the following, stated as Theorem \ref{mainbeta}(v) in the introduction:
\begin{cor} \label{depthhtopyinvt} If $H_0,H_1$ are two nondegenerate Hamiltonians such that $\phi_{H_0}^{1}=\phi_{H_1}^{1}$ and such that the paths $\phi_{H_0}^{t}$ and $\phi_{H_1}^{t}$ are homotopic rel endpoints in $Ham(M,\omega)$, then $\beta(H_0)=\beta(H_1)$.
\end{cor}

\begin{proof}[Proof of Corollary \ref{depthhtopyinvt}, assuming Lemma \ref{htopyinvt}]  Since adding a function of time to a Hamiltonian $H$ merely shifts the entire filtration on the Floer complex of $H$ by a constant and hence does not affect $\beta(H)$, there is no loss of generality in assuming that $H_0$ and $H_1$ are both normalized.  If $\beta>\beta(H_1)$, and if $c\in \partial_{H_0,J_0}(CF_*(H_0,J_0))\cap CF_{*}^{\lambda}(H_0,J_0)$, then $\Phi(c)\in
\partial_{H_1,J_1}(CF_*(H_1,J_1))\cap CF_{*}^{\lambda}(H_1,J_1)$, so there is $b\in CF_{*}^{\lambda+\beta}(H_1,J_1)$ 
with $\partial_{H_1,J_1}b=\Phi(c)$.  We will then have $\Phi^{-1}(b)\in CF_{*}^{\lambda+\beta}(H_0,J_0)$ and $\partial_{H_1,J_1}\Phi^{-1}(b)=c$.  This proves that $\beta(H_0)\leq \beta(H_1)$; the reverse inequality follows by reversing the roles of $H_0$ and $H_1$.
\end{proof}

\begin{proof}[Proof of Lemma \ref{htopyinvt}] 
Let $\{\psi_{s,t}\}_{(s,t)\in \mathbb{R}\times [0,1]}$ be a smooth family of Hamiltonian diffeomorphisms with $\psi_{s,t}=\phi_{H_0}^{t}$ for $s\leq 0$,  $\psi_{s,t}=\phi_{H_1}^{t}$ for $s\geq 1$, and $\psi_{s,0}=I$ while $\psi_{s,1}=\phi_{H_0}^{1}=\phi_{H_1}^{1}$ for all $s$.  For each $s\in \mathbb{R}$, there is then a unique normalized Hamiltonian $H_s\co  [0,1]\times M\to\mathbb{R}$ such that $\frac{d}{dt}(\psi_{s,t}(p))=X_{H_s}(t,\psi_{s,t}(p))$ for each $p\in M$.  

First of all we claim that there is no loss of generality in assuming that we have $H_s(0,\cdot)=H_s(1,\cdot)$, so that in fact the $H_s$ are well-defined and smooth on $(\mathbb{R}/\mathbb{Z})\times M$.  Indeed, take a smooth, monotone, surjective function $\chi\co [0,1]\to [0,1]$ such that $\chi'$ vanishes to infinite order at both $0$ and $1$.  Let $\zeta\co \mathbb{R}\times [0,1]\to [0,1]\times [0,1]$ be a smooth function, given by $\zeta(s,t)=(\eta(s),\chi_s(t))$ where $\chi_0(t)=\chi_1(t)=t$, each $\chi_s$ is smooth, monotone,  surjective, and satisfies $\chi^{(k)}_{s}(0)=\chi^{(k)}_{s}(1)$ for all $k\geq 1$, and  $\chi_{s}=\chi$ for $s\in [1/3, 2/3]$.  Further, $\eta\co \mathbb{R}\to [0,1]$ should be a smooth, monotone function with $\eta(s)=0$ for $s\leq 1/3$ and $\eta(s)=1$ for $s\geq 2/3$.  Then replacing $\psi_{s,t}$ by $\psi_{\zeta(s,t)}$ results in the Hamiltonians $H_s$ each satisfying $H_s(0,\cdot)=H_s(1,\cdot)$.  Accordingly we assume this to be true for the rest of the proof.

For $(s,t)\in\mathbb{R}\times [0,1]$ define the vector field $Y_{s,t}$ by $\frac{d}{ds}(\psi_{s,t}(p))=Y_{s,t}(\psi_{s,t}(p))$. 
As is well-known (see for instance the proof of Proposition II.3.3 in \cite{Ba}), $Y_{s,t}$ is a Hamiltonian vector field; let $K_{s}(t,\cdot)$ be the mean-zero function with $dK_{s,t}=\iota_{Y_{s,t}}\omega$.  Further, $\frac{\partial X_{H_s}(t)}{\partial s}-\frac{\partial Y_{s,t}}{\partial t}=[X_{H_s}(t),Y_{K_{s}}(t)]$ (\cite{Ba}, Proposition I.1.1).  So, where the Poisson bracket is defined by $\{H,K\}=\omega(X_H,X_K)$ and therefore satisfies $X_{\{H,K\}}=-[X_H,X_K]$, since Poisson brackets on closed manifolds always have mean zero (by Stokes' theorem) we obtain \begin{equation}\label{bracket} \frac{\partial H_s(t,\cdot)}{\partial s}-\frac{\partial K_{s}(t,\cdot)}{\partial t}=-\{H_s(t),K_s(t)\}.\end{equation}

Note that we have $K_{s}=0$ for $s\notin [0,1]$, and $K_s(0,\cdot)=K_s(1,\cdot)=0$ for all $s$.

To define the map $\Phi\co CF_*(H_0,J_0)\to CF_*(H_1,J_1)$ we now count, in the usual way, finite-energy, index-zero solutions $u\co\mathbb{R}\times(\mathbb{R}/\mathbb{Z})\to M$ to (a multivalued perturbation of) the equation 
\begin{equation}\label{twistedhtopy} \left(\frac{\partial u}{\partial s}-X_{K_s}(t,u(s,t))\right)+J_{s,t}\left(\frac{\partial u}{\partial t}-X_{H_{s}}(t,u(s,t))\right)=0\end{equation} for a suitable family of almost complex structures $J_{s,t}$ which coincides with $J_{0}^{t}$ for $s\ll 0$ and with $J_{1}^{t}$ for $s\gg 0$.  As noted in Section 2.2 of \cite{K}, this equation can be viewed as the equation for a $\tilde{J}$-holomorphic section $(s,t)\mapsto (s,t,u(s,t))$ of the trivial bundle $\mathbb{R}\times(\mathbb{R}/\mathbb{Z})\times M\to \mathbb{R}\times(\mathbb{R}/\mathbb{Z})$ where $\tilde{J}$ is a certain almost complex structure on the total space of the bundle, which is compatible with a certain symplectic form.  

We define the energy of a solution $u$ to (\ref{twistedhtopy}) as \[ E(u)=\int_{-\infty}^{\infty}\int_{0}^{1}\left|\frac{\partial u}{\partial s}-X_{K_s}(t,u(s,t))\right|^{2}_{J_{s,t}}dtds.\]  Since $X_{K_s}$ vanishes for $s\notin [0,1]$, the usual arguments show that a finite energy solution $u$ necessarily has $u(s,\cdot)\to \gamma^{\pm}$ uniformly and exponentially fast as $s\to\pm\infty$, where $\dot{\gamma}^-(t)=X_{H_0}(\gamma(t))$ and $\dot{\gamma}^+(t)=X_{H_1}(\gamma(t))$.  Furthermore, as pointed out in \cite{K}, if $[\gamma^+,w_+]=[\gamma^+,w_-\#u]$, then \[ \mathcal{A}_{H_0}([\gamma^-,w_-])-\mathcal{A}_{H_1}([\gamma^+,w_+])=E(u)+
\int_{-\infty}^{\infty}\int_{0}^{1}\left(\frac{\partial{H_s}}{\partial s}-\frac{\partial{K_s}}{\partial t}+\{H_s,K_s\}\right)(t,u(s,t))dtds.\]

But in our context the integrand in the last term above vanishes identically by (\ref{bracket}), and so we have \[ \mathcal{A}_{H_0}([\gamma^-,w_-])\geq \mathcal{A}_{H_1}([\gamma^+,w_+])\] whenever there is a finite-energy solution $u$ to (\ref{twistedhtopy}) with $u(s,\cdot)\to\gamma^{\pm}$ as $s\to\pm\infty$ and $[\gamma^+,w_+]=[\gamma^+,w_-\#u]$.  It follows directly from this that, where $\Phi\co CF_*(H_0,J_0)\to CF_*(H_1,J_1)$ is defined by counting finite-energy index-zero solutions to (\ref{twistedhtopy}) in the usual way, we have, for all $c\in CF_*(H_0,J_0)$, \begin{equation}\label{noninc} \mathcal{L}_{H_1}(\Phi(c))\leq \mathcal{L}_{H_0}(c).\end{equation}  It remains to show that equality holds above, and that $\Phi$ is an isomorphism of chain complexes.

Of course, standard arguments show that $\Phi$ is a chain map.  By instead counting solutions to (\ref{twistedhtopy}) with $H_{s}$ and $K_{s}$ replaced by $H_{1-s}$ and $K_{1-s}$, one obtains a chain map $\Psi\co CF_{*}(H_1,J_1)\to CF_{*}(H_0,J_0)$ which obeys $\mathcal{A}_{H_0}(\Psi(c))\leq \mathcal{A}_{H_1}(c)$ for all $c\in CF_*(H_1,J_1)$.  The composition $\Psi\circ\Phi\co CF_*(H_0,J_0)\to CF_*(H_0,J_0)$ will, by the usual gluing arguments, be equal to a map which (for sufficiently large $R$) counts finite-energy solutions $u$ to (a perturbation of) \[ \left( \frac{\partial u}{\partial s}-X_{\tilde{K}_s}(t,u(s,t))\right)+J_{s,t}\left(\frac{\partial u}{\partial t}-X_{\tilde{H}_{s}}(t,u(s,t))\right)=0,\] where now $\tilde{H}_s(t,\cdot)$ and $\tilde{K}_s(t\cdot)$ are the normalized Hamiltonians generating, respectively, the vector fields $\frac{\partial\tilde{\psi}_{s,t}}{\partial t}$ and $\frac{\partial\tilde{\psi}_{s,t}}{\partial s}$, and the symplectomorphisms $\tilde{\psi}_{s,t}$ are given by: \begin{align*} \tilde{\psi}_{s,t}&=\psi_{s+R,t} \mbox{ for $s<0$}\\ \tilde{\psi}_{s,t}&=\psi_{R-s,t} \mbox{ for $s\geq 0$}.\end{align*}
In particular for $|s|\geq R$ we have $\tilde{K}_s=0$ and $\tilde{H}_s=H_0$.

Now let $\tilde{\psi}_{s,t}^{\tau}$ be a smooth family of symplectomorphisms, parametrized by $s\in\mathbb{R}$, $t\in [0,1]$, and $\tau\in[0,1]$, such that $\tilde{\psi}_{s,t}^{0}=\tilde{\psi}_{s,t}$ and $\tilde{\psi}_{s,t}^{1}=\phi_{H_0}^{t}$ for all $s$.  This determines the normalized Hamiltonians $\tilde{H}_{s}^{\tau}(t,\cdot)$
and $\tilde{K}_{s}^{\tau}(t,\cdot)$, generating respectively $\frac{\partial{\tilde{\psi}_{s,t}^{\tau}}}{\partial t}$ and $\frac{\partial{\tilde{\psi}_{s,t}^{\tau}}}{\partial s}$.  In particular $\tilde{H}_{s}^{\tau}=H_0$ for all $|s|\geq R$, and all $\tau$, while $\tilde{H}_{s}^{1}=H_0$ for all $s$. Further we have, for all $\tau$, \[ \frac{\partial{\tilde{H}_{s}^{\tau}}}{\partial s}-\frac{\partial{\tilde{K}_{s}^{\tau}}}{\partial t}=-\{\tilde{H}_{s}^{\tau},\tilde{K}_{s}^{\tau}\}\] by the same arguments from \cite{Ba} that were used above.  Then where $J_{s,t,\tau}$ is a generic family of almost complex structures with $J_{s,t,0}=J_{s,t}$ and $J_{s,t,1}=J_{0}^{t}$, we define a map $\mathcal{K}\co CF_*(H_0,J_0)\to CF_*(H_0,J_0)$ by counting, as $\tau$ varies through $[0,1]$,  solutions $u\co\mathbb{R}\times(\mathbb{R}/\mathbb{Z})\to M$ to a perturbation of \[ \left(\frac{\partial u}{\partial s}-X_{\tilde{K}_{s}^{\tau}}(t,u(s,t))\right)+J_{s,t,\tau}\left(\frac{\partial u}{\partial t}-X_{\tilde{H}_{s}^{\tau}}(t,u(s,t))\right)=0,\] such that \[ E_{\tau}(u):=\int_{-\infty}^{\infty}\int_{0}^{1}\left|\frac{\partial u}{\partial s}-X_{\tilde{K}_{s}^{\tau}}(t,u(s,t))\right|^{2}_{J_{s,t,\tau}}dtds<\infty.
\]  
Any such solution with $u(s,\cdot)\to \gamma^{\pm}$ as $s\to\pm\infty$ and $[\gamma^+,w_+]=[\gamma^+,w_-\#u]$ has \begin{align*} \mathcal{A}_{H_0}([\gamma^-,w_-])-\mathcal{A}_{H_0}([\gamma^+,w_+])&=E_{\tau}(u)+\int_{-\infty}^{\infty}\int_{0}^{1}\left(\frac{\partial{\tilde{H}_{s}^{\tau}}}{\partial s}-\frac{\partial{\tilde{K}_{s}^{\tau}}}{\partial t}+\{\tilde{H}_{s}^{\tau},\tilde{K}_{s}^{\tau}\}\right)(t,u(s,t))dtds\\&=E_{\tau}(u)\geq 0.\end{align*}

Hence the resulting map $\mathcal{K}\co CF_*(H_0,J_0)\to CF_*(H_0,J_0)$ satisfies $\mathcal{L}_{H_0}(\mathcal{K}c)\leq \mathcal{L}_{H_0}(c)$ for all $c$.  By a standard argument, we will have \[ I-\Psi\circ\Phi=\partial_{H_0,J_0}\mathcal{K}+\mathcal{K}\partial_{H_0,J_0}\] where $I$ is the identity.  So since $\mathcal{L}_{H_0}(\partial_{H_0,J_0} c)<\mathcal{L}_{H_0}(c)$ for all $c$, it follows that $\Psi\circ\Phi=I+A$ where $A=-\mathcal{K}\partial_{H_0,J_0}-\partial_{H_0,J_0}\mathcal{K}$ satisfies, for some $\delta>0$, $\mathcal{L}_{H_0}(A c)\leq \mathcal{L}_{H_0}(c)-\delta$ for all $c$.  But then $B:=\sum_{k=0}^{\infty}(-A)^k$ gives a well-defined automorphism of $CF_*(H_0,J_0)$ inverting $I+A=\Psi\circ\Phi$, which proves that $B\circ\Psi$ is a left inverse for $\Phi$.  The same reasoning with $\Psi$ and $\Phi$ reversed produces a right inverse for $\Phi$ and so proves that $\Phi$ is an isomorphism of chain complexes.  Further, recalling that filtration levels are nonincreasing under the maps $\Psi$ and $A$ (hence also under $B$), if $c\in CF_*(H_0,J_0)$ we have \begin{align*}  \mathcal{L}_{H_0}(c)&=\mathcal{L}_{H_0}(B(\Psi(\Phi(c))))\leq \mathcal{L}_{H_0}(\Psi(\Phi(c)))\\&\leq \mathcal{L}_{H_1}(\Phi(c)).\end{align*}

Since we have already established the reverse inequality  (\ref{noninc}), this completes the proof.  
\end{proof}

\begin{remark}  Let us give what is  perhaps a more intuitive explanation of why $\Phi$ has the stated properties, avoiding the construction of $\Psi$ and of the chain homotopy $\mathcal{K}$.  Note that  there is a natural bijection $T\co CF_*(H_0,J_0)\to CF_*(H_1,J_1)$, defined as follows.  The $1$-periodic orbits of $X_{H_i}$ ($i=0,1$) are precisely given by, as $p$ ranges over $Fix(\phi_{H_0}^{1})=Fix(\phi_{H_1}^{1})$, setting $\gamma_{p}^{i}(t)=\phi_{H_i}^{t}(p)$.  Then where $\psi_{s,t}$ $((s,t)\in\mathbb{R}\times(\mathbb{R}/\mathbb{Z}))$ are as above, for $p\in Fix(\phi_{H_0}^{1})$ define $u_p\co \mathbb{R}\times (\mathbb{R}/\mathbb{Z})\to M$ by \[ u_p(s,t)=\psi_{s,t}(p).\]  Thus $u_p(s,\cdot)=\gamma_{p}^{0}$ for $s\leq 0$ and $u_p(s,\cdot)=\gamma_{p}^{1}$ for $s\geq 1$.  Now define $T\co CF_*(H_0,J_0)\to CF_*(H_1,J_1)$ by extending linearly from $T[\gamma_{p}^{0},w_0]=[\gamma_{p}^{1},w_0\#u_p]$.  Notice that $u_p(s,t)=\psi_{s,t}(p)$ solves (\ref{twistedhtopy}), and that $E(u_p)=0$, in view of which we have $\mathcal{A}_{H_0}([\gamma_{p}^{0},w_0])=[\gamma_{p}^{1},w_0\#u_p]$.  Thus $\mathcal{L}_{H_1}(Tc)=\mathcal{L}_{H_0}(c)$ for all $c$.  Now the only zero-energy solutions to (\ref{twistedhtopy}) are the $u_p$.  So at least modulo issues of sign and transversality, one expects $\Phi\co CF_*(H_0,J_0)\to CF_*(H_1,J_1)$ to have the form $\Phi=T+T'$ where $\mathcal{L}_{H_1}(T'c)<\mathcal{L}_{H_0}(c)$.  Such a map obviously preserves the filtration, and is easily seen to be invertible using the standard geometric series trick.
\end{remark}

We can now finally prove the result stated as Theorem \ref{mainbeta}(iv) in the introduction:   
\begin{cor}\label{dispbeta}Let $H,K\co(\mathbb{R}/\mathbb{Z})\times M\to\mathbb{R}$ be nondegenerate Hamiltonians with either $K\leq 0$ everywhere or $K\geq 0$ everywhere, let $S\subset M$ be a compact subset such that $\phi_{H}^{1}(S)\cap S=\varnothing$, and suppose that $|K(t,x)|\leq \ep$ for each $t\in\mathbb{R}/\mathbb{Z}$ and each $x\in M\setminus S$. Then 
\[ \beta(K)\leq 2\|H\|+2\ep.\]
\end{cor}

\begin{proof}  Let $\delta>0$. Let $\chi\co [0,1]\to [0,1]$ be a smooth, monotone, surjective function such that $\chi'$ vanishes to infinite order at both $0$ and $1$, such that $\chi(t)=0$ for $0\leq t\leq 1/2$ and such that $\chi'(t)\leq 2+\delta$ for all $t$.  Set $\tilde{K}(t)=\chi'(t,x)K(\chi(t),x)$. $\tilde{K}$ defines a smooth Hamiltonian on $(\mathbb{R}/\mathbb{Z})\times M$, with $\phi_{\tilde{K}}^{t}=\phi_{K}^{\chi(t)}$.  $\{\phi_{K}^{\chi(t)}\}_{0\leq t\leq 1}$ is homotopic rel endpoints to $\{\phi_{K}^{t}\}_{0\leq t\leq 1}$, so by Corollary \ref{depthhtopyinvt} \[ \beta(K)=\beta(\tilde{K}).\]  On the other hand we have $|\tilde{K}(t,x)|\leq (2+\delta)\ep$ whenever $x\notin S$, and $\tilde{K}(t,x)=0$ for $0\leq t\leq 1/2$, so by Lemma \ref{disp1} \[ \beta(\tilde{K})\leq \|H\|+(2+\delta)\ep.\]
We've thus shown that, for all $\delta>0$, $\beta(K)\leq 2\|H\|+(2+\delta)\ep$, from which the Corollary immediately follows.
\end{proof} 

\section{Spectral invariants and solutions to the Floer equation}\label{spectralsub}  As we have alluded to before, for any nondegenerate Hamiltonian $H$ there is a canonical isomorphism \[ \Phi_H\co H^{*}(M;\mathbb{Q})\otimes \Lambda_{\Gamma,\omega}\to HF_*(H);\] see \cite{PSS} for the construction.  This canonical isomorphism allows one to associate to any class $a\in  H^{*}(M;\mathbb{Q})\otimes \Lambda_{\Gamma,\omega}$ and any nondegenerate $H$ the \emph{spectral number} \[ \rho(H;a)=\inf\{\mathcal{L}_{H}(c)|c\in CF_{*}(H,J),\,[c]=\Phi_H(a)\}\] (where ``$[c]=\Phi_H(a)$'' means that $c$ is a cycle in the Floer complex with homology class $\Phi_H(a)$, and where $J\in\mathcal{J}^{reg}(H)$; $\rho(H;a)$ is independent of the choice of this $J$).  These spectral numbers are by now rather well-studied; see \cite{Oh1} for a detailed survey of their properties.  We will just be using the following results from the literature:

\begin{theorem}[Theorem I.5, \cite{Oh1}]\label{cts} Given $a\in  H^{*}(M;\mathbb{Q})\otimes \Lambda_{\Gamma,\omega}$, the function $H\mapsto \rho(H;a)$ extends to a function on the space of all (possibly degenerate) Hamiltonians, satisfying \[ |\rho(H;a)-\rho(K;a)|\leq
\|H-K\|\] for any two Hamiltonians $H$ and $K$.\end{theorem}

\begin{theorem}[Proposition 4.2, \cite{Ohlength}, Proposition 5.2, \cite{KL}] \label{cycle} Suppose that $K\co (\mathbb{R}/\mathbb{Z})\times M\to\mathbb{R}$ is nondegenerate,  that there is a point $p\in M$ such that each function $K(t,\cdot)$ attains a strict global minimum at $p$, and such that for each $t$ the Hessian $\nabla(\nabla K(t,\cdot))$ at $p$ is nondegenerate and satisfies $\|\nabla(\nabla K(t,\cdot))(p)\|<1$.  Denote by $\gamma_p$ the constant orbit of $X_K$ at $p$, and $w_p\co D^2\to M$ the constant disc at $p$, so that $\mathcal{A}_H([\gamma_p,w_p])=-\int_{0}^{1}K(t,p)dt$.  Then the class $\Phi_{K}(1)\in HF_*(K)$ has a representative $c$ of the form \[  c=[\gamma_p,w_p]+c'\] where \[ \mathcal{L}_{K}(c')<-\int_{0}^{1}K_t(p)dt.\] 
\end{theorem}

\begin{theorem}[Proposition 3.1, \cite{U3}] \label{disp}  Suppose that $H,K\co(\mathbb{R}/\mathbb{Z})\times M\to\mathbb{R}$ are two Hamiltonians with $K\leq 0$, and that there is a compact set  $S\subset M$ such that $supp(K(t,\cdot))\subset S$ for all $t\in\mathbb{R}/\mathbb{Z}$, while $\phi_{H}^{1}(S)\cap S=\varnothing$.  Then \[ \rho(K;1)\leq \|H\|.\]
\end{theorem}

(Of course, the Hamiltonian $K$ in Theorem \ref{disp} is necessarily degenerate, and we interpret the term $\rho(K;1)$ via Theorem \ref{cts}.)

These facts, together with what we have already done, have the following consequence:

\begin{prop} \label{exist1} Suppose that: \begin{itemize} \item $H,K\co(\mathbb{R}/\mathbb{Z})\times M\to\mathbb{R}$ are two Hamiltonians with $K$ satisfying the hypotheses of Theorem \ref{cycle} and $J\in\mathcal{J}^{reg}(K)$;  \item $\ep>0$ and $\delta>0$; \item $K(t,m)\leq 0$ for all $(t,m)$;\item there is a compact set  $S\subset M$ such that $\phi_{H}^{1}(S)\cap S=\varnothing$, while \[ |K(t,x)|\leq \ep \mbox{ for all }t\in\mathbb{R}/\mathbb{Z}\mbox{ and }x\in M\setminus S;\] and that \item where $p$ is the common global minimum for the $K(t,\cdot)$, we have \[\|H\|+\ep< -\int_{0}^{1}K(t,p)dt.\]\end{itemize} 

Then there is a generator $[\gamma,w]$ of $CF_{*}(K, J)$ and a solution $u\co(\mathbb{R}/\mathbb{Z})\times \mathbb{R}\to M$ to the Floer boundary equation \begin{equation}\label{bdryK} \frac{\partial u}{\partial s}+J^t(u(s,t))\left(\frac{\partial u}{\partial t}-X_{K}(t,u(s,t))\right)=0 \end{equation} satisfying \begin{itemize} \item $u(s,\cdot)\to \gamma$ as $s\to -\infty$, while $u(s,\cdot)\to\gamma_p$ as $s\to\infty$.
\item $[\gamma_p,w_p]=[\gamma_p,w\#u_p]$, and \item \[ 0<E(u)=\mathcal{A}_{K}([\gamma,w])-\mathcal{A}_{K}([\gamma_p,w_p])< 2\|H\|+2\ep+\delta.\]\end{itemize}
\end{prop}

\begin{proof}  Since $S$ is compact and $\phi_{H}^{1}(S)\cap S=\varnothing$, we can find a smooth function $\chi\co M\to[0,1]$ such that $\phi_{H}^{1}(supp\,\chi)\cap (supp\,\chi)=\varnothing$ and $\chi|_S=1$.  Put $K'(t,m)=\chi(m)K(t,m)$; the support of $K'(t,\cdot)$ is contained in $supp\,\chi$ for each $t$, so Proposition \ref{disp} applies to show that \[ \rho(K';1)\leq \|H\|.\]  Meanwhile the hypothesis on $K$ ensures that $|K'(t,m)-K(t,m)|\leq \ep$ for each $t,m$, and so $\|K-K'\|\leq \ep$, whence by Theorem \ref{cts}, \[  \rho(K;1)\leq \|H\|+\ep.\]  In particular, the last hypothesis of the proposition shows that we have \[ \rho(K;1)< -\int_{0}^{1}K(t,p)dt=\mathcal{L}_{K}(c) \] where $c\in CF_{*}(K,J)$ is the chain of Theorem \ref{cycle}.  
Hence by the definition of the spectral number $\rho(K;1)$, there must be some other chain, say $d$, which (like $c$) represents  the class $\Phi_{K}(1)\in HF_{*}(K)$ and which has $\mathcal{L}_K(d)<\mathcal{L}_K(c)$.  

Since $c$ and $d$ are homologous in $CF_{*}(K)$, there is $b\in CF_{*}(K)$ such that \[ \partial_{K,J}b=c-d.\]  We have $c-d=[\gamma_p,w_p]+(c'-d)$ where $\mathcal{L}_{K}(c'-d)<\mathcal{A}_K([\gamma_p,w_p])$, in light of which $\mathcal{L}_K(c-d)=\mathcal{A}_K([\gamma_p,w_p])$.  

Since $c-d\in\partial_{K,J}(CF_{*}(K,J))\cap CF_{*}^{\mathcal{A}_K([\gamma_p,w_p])}(K,J)$, it follows from the definition of $\beta(K)$ that the element $b\in CF_{*}(K,J)$ such that $\partial_{K,J}b=c-d$ can be chosen in such a way that $\mathcal{L}_K(b)<\mathcal{A}_K([\gamma_p,w_p])+\beta(K)+\delta$.  Now writing $b=\sum b_i[\gamma^i,w^i]$ (where each $\mathcal{A}_K([\gamma^i,w^i])<\mathcal{A}_K([\gamma_p,w_p])+\beta(K)+\delta$), since $[\gamma_p,w_p]$ appears with coefficient $1$ in $\partial_{K,J}b$  it follows from the definition of $\partial_{K,J}$ that, for some $i$, there is a solution $u$ to (\ref{bdryK}) having $u(s,\cdot)\to \gamma^i$ as $s\to-\infty$, $u(s,\cdot)\to \gamma_p$ as $s\to\infty$, and  $[\gamma_p,w_p]=[\gamma_p,w^i\#u]$.  This $u$ is the desired solution; we have \[ E(u)=\mathcal{A}_K([\gamma^i,w^i])-
\mathcal{A}_K([\gamma_p,w_p])<\beta(K)+\delta,\] so the fact that (by Corollary \ref{dispbeta}) we have $\beta(K)\leq 2\|H\|+2\ep$ implies the result (with $[\gamma,w]=[\gamma^i,w^i]$).
\end{proof}

\section{From non-degenerate Hamiltonians to degenerate ones}\label{nondeg-deg}

The Hamiltonians that we have considered thus far have all been nondegenerate; however, our applications will all depend on finding suitable periodic orbits of certain degenerate Hamiltonians, which in particular will be supported within a small open set $W$.  Not surprisingly, we will pass from the nondegenerate to the degenerate case via compactness arguments.  It turns out to be  useful to achieve this in two steps rather than one: first, pass from a nondegenerate Hamiltonian to a Hamiltonian which is has support in $W$ (hence is degenerate) but whose restriction to $W$ is nondegenerate; second, pass from this intermediate Hamiltonian to the (even more degenerate) one that we are interested in.  The purpose of separating the procedure into two steps is that doing so enables us to guarantee that the periodic orbits and solutions to the Floer boundary  equation that we obtain are contained in an appropriate open neighborhood $V$ of the support of the Hamiltonian.  This accounts for the local nature of the hypotheses of some of the theorems stated in the introduction.    

The following Proposition implements the first step of the procedure.

\begin{prop}\label{energypartial} Fix an almost complex structure $J_0$ on $M$ compatible with $\omega$, and measure distances using the metric $g(v,w)=\omega(v,J_0w)$.  Let $U,V\subset M$ be open sets, with smooth (possibly empty) orientable boundaries $\partial U,\partial V$, such that $\bar{U}\subsetneq V$.  Then there is a constant $c_{UV}$, depending only on $U,V,\omega|_V$, and $J_0|_V$, with the following property.  Suppose that $K\co (\mathbb{R}/\mathbb{Z})\times M\to\mathbb{R}$ is a Hamiltonian such that \begin{itemize}
\item $K(t,x)\leq 0$ for all $(t,x)\in(\mathbb{R}/\mathbb{Z})\times M$.
\item For some open set $W$ with $\bar{W}\subset U$, $K^{-1}((-\infty,0))=(\mathbb{R}/\mathbb{Z})\times W$.
\item There are finitely many $1$-periodic orbits of $X_K$ that are contained in $W$, and each of these is nondegenerate.
\item There is $p\in W$ such that $p$ is a strict global minimum of each $K(t,\cdot)$, and  $\nabla(\nabla K(t,\cdot))$ is nondegenerate and satisfies $\|\nabla(\nabla K(t,\cdot))(p)\|<1$.
\item For some Hamiltonian $H\co (\mathbb{R}/\mathbb{Z})\times M\to \mathbb{R}$ with $\|H\|<c_{UV}$, we have $\phi_{H}^{1}(\bar{W})\cap \bar{W}=\varnothing$.
\item $\|H\|<-\int_{0}^{1}K(t,p)dt$.  
\end{itemize}

Then there exists a solution $u\co \mathbb{R}\times(\mathbb{R}/\mathbb{Z})\to \overline{V}$ to the equation \begin{equation}\label{epeqn} \frac{\partial u}{\partial s}+J_0(u(s,t))\left(\frac{\partial u}{\partial t}-X_K(t,u(s,t))\right)=0\end{equation} and a periodic orbit $\gamma\co \mathbb{R}/\mathbb{Z}\to \bar{V}$ such that \begin{itemize}\item $u(s,\cdot)$ is partially asymptotic to $\gamma$ as $s\to -\infty$, which is to say, there is a sequence $s_j\to -\infty$ such that $\gamma(s_j,\cdot)\to\gamma$ uniformly as $j\to\infty$; \item $u(s,\cdot)\to\gamma_p$ uniformly as $s\to\infty$.
\item $0<E(u)\leq 2\|H\|$.\end{itemize}
\end{prop}
\begin{remark} We emphasize that we are proving that $u$ has image contained in $\bar{V}$.   
\end{remark}
\begin{proof}  We begin by identifying the constant $c_{UV}$.  $c_{UV}$ will be equal to the minimum of: (i) one-half of the minimal energy of a $J_0$-holomorphic sphere in $\bar{V}$; and (ii) the constant $c'_{UV}$ of the following Lemma.
\begin{lemma}\label{monotonicity} If $U$, $V$, and $J_0$ are as in Proposition \ref{energypartial}, there is a constant $c'_{UV}$ with the following property.  Let $u\co \Sigma\to M$ be a smooth map from a connected Riemann surface $(\Sigma,j)$ with the property that \[ du+J_0(u)\circ du\circ j\mbox{ vanishes identically on }u^{-1}(\bar{V}\setminus U),\] and suppose that \[ u^{-1}(U)\neq\varnothing\mbox{ and }u^{-1}(M\setminus \bar{V})\neq\varnothing.\]  Then \[ \frac{1}{2}\int_{\Sigma}|du|^{2}_{J_0}\geq 2c'_{UV}.\] 
\end{lemma}

\begin{remark}\label{VM}Note that Lemma \ref{monotonicity} is true for trivial reasons if $M=\bar{V}$, with $c'_{UV}=\infty$.  So in case $\bar{V}=M$, the constant $c_{UV}$ of Proposition \ref{energypartial} will be equal to one-half of the minimal energy of a $J_0$-holomorphic sphere in $M$.  In the proof below we accordingly assume that $\bar{V}\neq M$, in view of which the hypotheses of the Lemma imply that $\partial U$ and $\partial V$ are nonempty. \end{remark}

\begin{proof}  For $x,y\in M$ let $d(x,y)$ denote the distance from $x$ to $y$ as measured by the Riemannian metric $g$ induced by $\omega$ and $J_0$.    Let $\beta\co M\to\mathbb{R}$ be a smooth function (constructed for instance with the aid of standard collar neighborhoods $\partial U\times (-\ep,\ep),\partial V\times (-\ep,\ep)$) such that $U=\beta^{-1}(-\infty,0)$, $V=\beta^{-1}(-\infty,1)$, and $0$ and $1$ are regular values of $\beta$ with $\beta^{-1}(\{0\})=\partial U$ and $\beta^{-1}(\{1\})=\partial V$.  Let \[ r_0=\min\{d(x,y)|x\in \beta^{-1}(\{1/2\}),\,y\in(\partial U)\cup(\partial V)\}\] and let $r_1$ be the minimum of $r_0$ and the injectivity radius of the Riemannian manifold $(M,g)$.

Our hypothesis implies that $\beta\circ u$ has image meeting both $(-\infty,0)$ and $(1,\infty)$, so since $\Sigma$ is connected $1/2\in Im(\beta\circ u)$.  So choose $z_0\in \Sigma$ such that $\beta(u(z_0))=1/2$.  Let $r_{2}^{2}$ be any regular value of the function $w\co \Sigma\to M$ defined by $z\mapsto d(u(z_0),u(z))^{2}$ with the property that $r_{1}^{2}/2<r_{2}^{2}<r_{1}^{2}$.  Let $S=w^{-1}([0,r_{2}^{2}])$.  In particular $u(S)\subset V\setminus \bar{U}$.  $S$ is then a submanifold with boundary of $\Sigma$, and $u|_{S}\co S\to \bar{V}\setminus U$ is a $J_0$-holomorphic map such that $u(z_0)\in u(S)$ and $u(\partial S)$ is contained in the boundary of the ball of radius $r_2$ around $z_0$. So Proposition 4.3.1(ii) of \cite{Si} implies that there is a constant $C$ (depending only on $\omega|_{V\setminus\bar{U}}$ and $J_0|_{V\setminus\bar{U}}$) such that \[ \int_{S}|du|^{2}\geq Cr_{2}^{2}\geq Cr_{1}^{2}/2.\]

Accordingly we may set $c'_{UV}=Cr_{1}^{2}/4$.
\end{proof}

We now return to the Proof of Proposition \ref{energypartial}.   For each $i\in\mathbb{N}$ let $K^{i}\co (\mathbb{R}/\mathbb{Z})\times M\to\mathbb{R}$ be a nondegenerate Hamiltonian with $\|K^{i}-K\|_{C^2}<\frac{1}{i}$ and with $K^{i}(t,\cdot)$ having a strict global minimum at $p$.  Let $J_{i}\in\mathcal{J}^{reg}(K^i)$, with the associated paths $\{J_{i}^{t}\}_{0\leq t\leq 1}$ of almost complex structures satisfying $\|J^{t}_{i}-J_0\|_{C^2}<\frac{1}{i}$ for each $t\in\mathbb{R}/\mathbb{Z},i\in\mathbb{N}$.  Note that $|K^i(t,x)|\leq 1/i$ for $x\notin \bar{W}$.   For $i$ large enough, we will have $0<\|\nabla(\nabla K^i(t,\cdot))(p)\|<1$, and \[ \|H\|+\frac{1}{i}<-\int_{0}^{1}K^{i}(t,p)dt.\]  So since $\phi_{H}^{1}(\bar{W})\cap \bar{W}=\varnothing$, Proposition \ref{exist1} produces a solution $u^i\co\mathbb{R}\times(\mathbb{R}/\mathbb{Z})\to M$ to \begin{equation}\label{approx1} \frac{\partial u^i}{\partial s}+J^{t}_{i}(u^i(s,t))\left(\frac{\partial u^i}{\partial s}-X_{K^{i}}(t,u^i(s,t))\right)=0\end{equation} having \[ 0<E(u^{i})<2\|H\|+\frac{2}{i},\] and \[ u(s,\cdot)\to \gamma_p\] uniformly as $s\to\infty$.

Choose a small ball $B_{\delta}(p)$ around $p$; in particular $\overline{B_{\delta}(p)}$ should be contained in $W$ and should miss the other periodic orbits of $X_K$ and of the $X_{K^i}$.  Note that none of the $u^i$ have image contained entirely in $B_{\delta}(p)$ (for any number of reasons, for instance because it's easy to see that otherwise $u^{i}$ could not have positive energy). Let \[ T_i=\inf\{s\in\mathbb{R}|u^{i}(s',t)\in B_{\delta}(p)\mbox{ for all }t\in\mathbb{R}/\mathbb{Z}\mbox{ and all }s'\geq s\}.\]  Now put \[ \tilde{u}^i(s,t)=u^i(s+T_i,0).\]  Then the $\tilde{u}^{i}$ still solve (\ref{approx1}) and have $\tilde{u}^{i}(s,\cdot)\to \gamma_p$ as $s\to\infty$ and $E(\tilde{u}^{i})<2(\|H\|+1/i)$.  Also, \[ \tilde{u}^{i}(\{0\}\times\mathbb{R}/\mathbb{Z})\cap \partial B_{\delta}(p)\neq\varnothing,\] while \[ \tilde{u}^{i}((0,\infty)\times\mathbb{R}/\mathbb{Z})\subset \overline{B_{\delta}(p)}.\]

Gromov compactness applied to the $\tilde{u}^{i}$ produces a solution $u$ to (\ref{epeqn}) to which (after passing to a subsequence) the $\tilde{u}^{i}$ converge modulo bubbling, with \[ E(u)\leq \overline{\lim}E(\tilde{u}^{i})=2\|H\|<2c_{UV}.\]  We claim that any bubbles must not intersect $U$.  Indeed, the $E(u_i)$ all eventually have energy less than $2c_{UV}$, which is less than or equal to $2c'_{UV}$ (so that no bubble can meet both $U$ and $M\setminus \bar{V}$ by Lemma \ref{monotonicity}) and is also less than or equal to the minimal energy of a $J_0$-holomorphic sphere in  $\bar{V}$ (so that no bubble can be contained in $\bar{V}$).  So since $\tilde{u}^i([0,\infty)\times\mathbb{R}/\mathbb{Z})\subset \overline{B_{\delta}(p)}\subset U$, the $\tilde{u}^i$ converge uniformly on compact subsets of $[0,\infty)\times (\mathbb{R}/\mathbb{Z})$.  In particular, we have $u(0,t)\in \partial(B_{\delta}(p))$ for some $t$ (since the same statement holds for the $\tilde{u}^{i}$). Further, since $\gamma_p$ is the only periodic orbit of $X_K$ in $\overline{B_{\delta}(p)}$, the facts that $E(u)<\infty$ and that $u([0,\infty)\times (\mathbb{R}/\mathbb{Z}))\subset  \overline{B_{\delta}(p)}$ force $u(s,\cdot)\to \gamma_p$ uniformly as $s\to\infty$.  In particular $u$ is not constant (since its image meets $\partial(B_{\delta}(p))$ while it is asymptotic to $p$), so $E(u)>0$.  The fact that $E(u)<\infty$ forces there to exist a $1$-periodic orbit $\gamma$ of $X_K$ to which $u$ is partially asymptotic as $s\to -\infty$ (this follows easily along the lines of the hint to Exercise 1.22 of \cite{Sal}; details are left to the reader).  See Section 5.2 of \cite{G06} for a more detailed treatment of a similar case. 

Finally, the image of $u$ meets $U$, and $u|_{u^{-1}(\bar{V}\setminus U)}$ is $J_0$-holomorphic, so the fact that $E(u)\leq 2\|H\|<2c'_{UV}$ implies, by Lemma \ref{monotonicity}, that the image of $u$ cannot meet $M\setminus \bar{V}$.  This completes the proof that $u$ has the desired properties.
\end{proof}

The Hamiltonians that we are ultimately interested in will satisfy the following condition:

 \begin{definition}\label{flat} A Hamiltonian $K\co (\mathbb{R}/\mathbb{Z})\times M\to\mathbb{R}$ is said to \emph{have a flat autonomous minimum at} $p\in M$ if \begin{itemize} \item[(i)] for each  $t\in\mathbb{R}/\mathbb{Z}$, $K(t,\cdot)$ has a global minimum at $p$; \item[(ii)] There is an open neighborhood $G$ of $p$ such that, for each $m\in G$, $K(t,m)$ is independent of $t$.
 \item[(iii)] The set $S=\{m\in M|K(t,m)=K(0,p)\mbox{ for all }t\in\mathbb{R}/\mathbb{Z}\}$ is a proper compact subset of $G$, and $\nabla(\nabla K(t,\cdot))=0$ at every point of $S$. \end{itemize} 
 \end{definition} 

Our applications will be consequences of the following:

\begin{prop}\label{flatlemma} Let $U,V\subset M$ be open subsets as in Proposition \ref{energypartial} (so in particular $\bar{U}\subsetneq V$).    Let $K\co (\mathbb{R}/\mathbb{Z})\times M\to (-\infty,0]$ be a not-identically-zero Hamiltonian having a flat autonomous minimum at $p$, such that for some open set $W$ with smooth boundary and $\bar{W}\subset U$, we have $K^{-1}((-\infty,0))\subset (\mathbb{R}/\mathbb{Z})\times W$.  Assume further that, for some Hamiltonian $H\co (\mathbb{R}/\mathbb{Z})\times M\to\mathbb{R}$, we have \begin{itemize}\item[(i)] $\phi_{H}^{1}(\bar{W})\cap \bar{W}=\varnothing$; \item[(ii)] $\|H\|<c_{UV}$, where $c_{UV}$ is as in Proposition \ref{energypartial}; \item[(iii)] $\|H\|<-K(0,p)$; and \item[(iv)] Where\footnote{We use the convention that the infimum of the empty set is $\infty$.} \[\lambda_0(V)=\inf\left((0,\infty)\cap \{\int_{S^2} w^*\omega|w\in C^{\infty}(S^2,\bar{V})\}\right),\] it holds that \[ \lambda_0(V)>2\|H\|+\|K\|.\]\end{itemize}  Then there is a \textbf{nonconstant} $1$-periodic orbit $\gamma$ of $X_K$; a point $q\in M$ such that $K(t,q)=K(0,p)$ for each $t$; and a solution $u\co \mathbb{R}\times (\mathbb{R}/\mathbb{Z})\to \bar{V}$ to the equation \begin{equation}\label{epeqn2} \frac{\partial u}{\partial s}+J_0\left(\frac{\partial u}{\partial t}-X_K(t,u(s,t))\right)=0\end{equation} such that \begin{itemize} \item[(i)] $u$ is partially asymptotic to the nonconstant orbit $\gamma$ as $s\to -\infty$; \item[(ii)] $u$ is partially asymptotic to the constant orbit $\gamma_q$ as $s\to\infty$; and \item[(iii)] $0<E(u)\leq 2\|H\|$.\end{itemize}
\end{prop}

\begin{proof}  This follows from Proposition \ref{energypartial} by a compactness argument similar to ones that we have made previously.  Let $g\co M\to [-1,-1/2]$ be a Morse function having a strict global minimum at $p$, and let $\chi\co M\to [0,1]$ be a smooth function with $\chi^{-1}((0,1])=W$ and $\chi(p)=1$.  For $i\in\mathbb{N}$, the function $K^{i}_{0}(t,m)=K(t,m)+\frac{1}{i}\chi(m)g(m)$ will have a strict global minimum at $p$ with a nondegenerate Hessian there, and will have $K^{i}_{0}((-\infty,0))=(\mathbb{R}/\mathbb{Z})\times W$.

Let $S$ and $G$ be sets with the properties indicated in Definition \ref{flat}.

  Once $i$ is sufficiently large, suitable perturbations $K^i$ of $K^{i}_{0}$ which have $\|K^i-K^{i}_{0}\|_{C^2}<1/i$ and which coincide with $K^{i}_{0}$ both on $M\setminus W$ and on a small neighborhood of $S$ will satisfy the hypotheses of Proposition \ref{energypartial}.  Proposition \ref{energypartial} therefore provides solutions $u^i$
to the $K^i$-version of (\ref{epeqn2}), partially asymptotic at $-\infty$ to a periodic orbit $\gamma^i$ of 
$X_{K^i}$, and partially asymptotic at $+\infty$ to $\gamma_p$, such that $0<E(u^i)\leq 2\|H\|$. 

Now, recalling that the Hessian of $K$ vanishes along $S$, for $i$ sufficiently large the restriction of the Hamiltonian vector field $X_{K^i}$ to the closure of a small neighborhood $G'$ of $S$ (with $\bar{G'}\subset G$, and with $K^i|_{(\mathbb{R}/\mathbb{Z})\times \bar{G'}}=K^{i}_{0}|_{(\mathbb{R}/\mathbb{Z})\times \bar{G'}}$) will be independent of $t$ and will be very small in $C^1$-norm.  An easy argument using the Yorke estimate \cite{Y} then shows that, if $G'$ is chosen appropriately (and independently of sufficiently large $i$), then the only $1$-periodic orbits of $X_{K^{i}}$ that are contained in $\overline{G'}$ will be constant orbits. 

Let $i$ be large enough that $\lambda_0(V)> 2\|H\|+\|K^i\|$.  Since $u^i$ has image contained in $V$, if $u^i$ were partially asymptotic as $s\to -\infty$ to the constant orbit $\gamma_q$ then we would have \[ E(u^i)\in \int_{0}^{1}(K^i(t,p)-K^i(t,q))dt+\lambda_0(V)\mathbb{Z},\] where if $\lambda_0(V)=\infty$ we interpret $\lambda_0(V)\mathbb{Z}=\{0\}$.  Since $p$ is a global minimum of the $K^i(t,\cdot)$, the set on the right hand side contains no positive number that is smaller than $\lambda_0(V)-\|K^i\|$, whereas the $u^i$ have $0<E(u^i)\leq 2\|H\|$.  But by our assumption $\lambda_0(V)-\|K^i\|>2\|H\|$, so the $u^i$ must \emph{not} be partially asymptotic as $s\to -\infty$ to any constant orbit, and in particular are not partially asymptotic as $s\to\ -\infty$ to any orbit contained in $\overline{G'}$.    

The map $u^i$, therefore, does not have image contained in $\overline{G'}$.  Let us define $T^i\in\mathbb{R}$ by \[ T^i=\inf\{s\in\mathbb{R}|u([s,\infty)\times(\mathbb{R}/\mathbb{Z}))\subset G'\}.\]  As usual setting $\tilde{u}^i(s,t)=u(s+T^i,t)$, Gromov compactness applied to the $\tilde{u}^{i}$ will produce a solution $u\co \mathbb{R}\times(\mathbb{R}/\mathbb{Z})\to \bar{V}$ of (\ref{epeqn2}), partially asymptotic as $s\to \infty$ to a periodic orbit of $X_K$ which is contained in $G'$, and hence must be equal to some $\gamma_q$ where $q\in S$.  ($u$ has image contained in $\bar{V}$ and no bubble passes through any point of $U$ because $E(\tilde{u}^i)\leq 2\|H\|<2c_{UV}$.)  Since the $\tilde{u}^i(\{0\}\times(\mathbb{R}/\mathbb{Z}))$ meet $\partial G'$, the same must be true for $u(\{0\}\times(\mathbb{R}/\mathbb{Z}))$, and so $u$ is not constant, and obeys $0<E(u)\leq 2\|H\|$.  Where $\gamma$ is any orbit to which $u$ is partially asymptotic as $s\to -\infty$, the fact that $0<E(u)\leq 2\|H\|$ while $\lambda_0(V)> 2\|H\|+\|K\|$ shows, via the same argument as in the previous paragraph, that $\gamma$ must not be constant.  

\end{proof}

\section{Infinitely many periodic points}\label{infsect}

We now restate and prove Theorem \ref{infcopy}.
\begin{theorem}\label{inf}  Let $U$ and $V$ be open subsets of the closed symplectic manifold $(M,\omega)$ with smooth orientable boundaries, and $\bar{U}\subset V$.  Assume that $[\omega]$ vanishes on $\pi_2(\bar{V})$.  Then the constant $c_{UV}$ of Proposition \ref{energypartial} obeys the following properties.  Let $K\co (\mathbb{R}/\mathbb{Z})\times M\to (-\infty,0]$ be a not-identically-zero Hamiltonian with a flat autonomous minimum at a point $p\in U$.  Assume that $K^{-1}((-\infty,0))\subset (\mathbb{R}/\mathbb{Z})\times W$, where $W\subset U$ has displacement energy $e(W,M)$ satisfying $e(W,M)<c_{UV}$.  Then the symplectomorphism $\phi_{K}^{1}$ has infinitely many geometrically distinct, nontrivial periodic points.\end{theorem}

\begin{proof}  For any $n\in\mathbb{Z}$, the $n$-fold composition $(\phi_{H}^{1})^n$ arises as the time-$1$ map of the Hamiltonian $K^{\#n}$ defined by $K^{\#n}(t,m)=nK(nt,m)$.  As this formula makes clear, the fact that $K$ has a flat autonomous minimum at $p\in U$ implies that each $K^{\#n}$ also has such a flat autonomous minimum.   Choose a Hamiltonian $H$ such that $\phi_{H}^{1}(\bar{W})\cap \bar{W}=\varnothing$ and $\|H\|<c_{UV}$.  For some $n_0\in\mathbb{N}$, once $n\geq n_0$ it will be true that $K^{\#n}$ and $H$ obey the hypotheses of Proposition \ref{flatlemma}.  

Hence for each $n\geq n_0$, there is a nonconstant $1$-periodic orbit $\gamma^{(n)}$ of $X_{K^{\#n}}$ and a disc $w^{(n)}\co D^2\to \bar{V}$ with $w^{(n)}(e^{2\pi it})=\gamma^{(n)}(t)$ such that \begin{equation}\label{actionrange} 0<\mathcal{A}_{K^{\#n}}([\gamma^{(n)},w^{(n)}])-\mathcal{A}_{K^{\#n}}([\gamma_{p},w_p])\leq 2\|H\| \end{equation} ($w^{(n)}$ can be constructed in a straightforward way from the map $u$ in the $K^{\#n}$ version of Proposition \ref{flatlemma}; note that where $q$ is the point identified in Proposition \ref{flatlemma} we have $\mathcal{A}_{K^{\#n}}([\gamma_{q},w_q])=\mathcal{A}_{K^{\#n}}([\gamma_{p},w_p])$).

%Notice that if $\gamma^{(n)}$ were a constant, say at $x$, then $w^{(n)}$ would be equivalent to a sphere, so $\int_{w^{(n)}}\omega$ would vanish by our hypothesis on $\pi_2(\bar{V})$.  But we would then have \[   
%\mathcal{A}_{K^{\#n}}([\gamma^{(n)},w^{(n)}])-\mathcal{A}_{K^{\#n}}([\gamma_{p},w_p])=\int_{0}^{1}(K(t,p)-K(t,x))dt\leq 0,\] %violating the first inequality in (\ref{actionrange}).  Thus none of the $\gamma^{(n)}$ is constant.

Now suppose, to get a contradiction, that $\phi_{K}^{1}$ had just finitely many geometrically distinct nontrivial periodic points, and let $n_1$ be any number which  both is larger than $n_0$ and is a common multiple of the minimal periods of each of these finitely many periodic points.  Let $\gamma_1,\ldots,\gamma_N$ be the nontrivial contractible $1$-periodic orbits of $X_{K^{\#n_1}}$.  Where $w_1,\ldots,w_N$ are discs in $\bar{V}$ such that $w_j(e^{2\pi it})=\gamma_j(t)$, let \[ A=\{\mathcal{A}_{K^{\#n_1}}([\gamma_j,w_j])-\mathcal{A}_{K^{\#n_1}}([\gamma_p,w_p])|j=1,\ldots,N\}.\]  For each $j$,  
$\mathcal{A}_{K^{\#n_1}}([\gamma_j,w_j])$ is of course independent of the choice of $w_j$ in light of our assumption on $\pi_2(\bar{V})$.  

$A$ is a finite set of real numbers.  Our assumption that $\phi_{K}^{1}$ has just finitely many geometrically distinct nontrivial periodic orbits implies that, for $k\geq 1$, the orbit $\gamma^{(kn_1)}$ (which we already know is nonconstant) is, up to time-shift, a $k$-fold iteration of one of the $\gamma_j$, and that \[ \mathcal{A}_{K^{\#kn_1}}([\gamma^{(kn_1)},w^{(kn_1)}])-\mathcal{A}_{K^{\#kn_1}}([\gamma_{p},w_p])
=k(\mathcal{A}_{K^{\#n_1}}([\gamma_j,w_j])-\mathcal{A}_{K^{\#n_1}}([\gamma_p,w_p])).\]  But we have \[ 0<\mathcal{A}_{K^{\#kn_1}}([\gamma^{(kn_1)},w^{(kn_1)}])-\mathcal{A}_{K^{\#n}}([\gamma_{p},w_p])\leq 2\|H\|\] by (\ref{actionrange}), whereas for $k$ sufficiently large the set $\{ka|a\in A\}$ will contain no numbers in the interval $(0,2\|H\|]$.  This contradiction proves the theorem. 
\end{proof}

\section{Coisotropic submanifolds}\label{geodco}

We consider now a closed symplectic manifold $(M,\omega)$ containing a closed coisotropic submanifold $N$; thus $(T_pN)^{\omega}\subset T_pN$ for each $p\in N$. Since $\omega$ is closed, the distribution $TN^{\perp_{\omega}}$ on $N$ is integrable; therefore the Frobenius theorem produces a foliation $\mathcal{F}$ (the \emph{characteristic foliation}) on $N$ whose tangent spaces are given by $TN^{\perp_{\omega}}$, and the rank of this foliation is equal to the codimension of $N$ in $M$.

Recent work of V. Ginzburg \cite{G06} and others, generalizing work of P. Bolle \cite{B}, has extended famous symplectic rigidity results about Lagrangian submanifolds or hypersurfaces in symplectic manifolds to certain classes of coisotropic submanifolds.  The coisotropic submanifolds $N\subset M$ in these results are assumed to be \emph{stable}, which is to say that, where $\dim M=2n$ and $\dim N=2n-k$, there should exist $\alpha_1,\ldots,\alpha_k\in \Omega^1(N)$ such that \begin{itemize} \item[(i)] $\ker(\omega|_N)\subset\ker \alpha_i$ for each $i$, and \item[(ii)] $\alpha_1\wedge\cdots\wedge\alpha_k\wedge (\omega|_N)^{n-k}$ is a volume form on $N$.\end{itemize}

Additional results are sometimes proven under the assumption that $N$ has contact type, which is to say that the above $\alpha_i$ additionally satisfy $d\alpha_i=\omega|_N$ for each $i$.  

In the case that $N$ is orientable and $k=1$, so that $N$ is a hypersurface, these definitions are consistent with the traditional ones.  Recall in particular that an orientable hypersurface $N\subset M$ is called \emph{stable} if for each embedding $\psi\co (-\ep,\ep)\times N\to M$ of a neighborhood of $N$ such that $\psi|_{\{0\}\times N}$ is the inclusion of $N$ into $M$, the characteristic foliations on the various $\psi(\{t\}\times N)$ are the same when viewed as foliations on $N$ for all sufficiently small $t$.  This is equivalent (see, \emph{e.g.}, \cite{EKP}) to the existence of $\alpha\in \Omega^1(N)$ such that $\alpha\wedge(\omega|_N)^{n-1}$ is a volume form and $\ker d\alpha\subset \ker \omega|_N$; for one direction, given $\psi\co (-\ep,\ep)\times N\to M$ inducing the same characteristic foliation on each $\{t\}\times N$, we can recover $\alpha$ by writing $\psi^*\omega=dt\wedge \alpha_t+\omega_t$ where $\alpha_t\in \Omega^1(N)$ and $\omega_t\in \Omega^2(N)$, for then we will have $\omega_0=\omega|_N$ and $\alpha=\alpha_0$ will be as required.

As noted in Remark 2.4 of \cite{G06}, the requirement of stability is a rather restrictive one to impose, especially when $k$ is not small.  Indeed, assume that $N$ is stable, set $\omega_0=\omega|_N\in\Omega^2(N)$, and let $\alpha_1,\ldots,\alpha_k\in\Omega^1(N)$ be as in the definition of stability.  $\ker\omega_0$ then has rank $k$, and so there are vector fields $X_1,\ldots X_k$ on $N$, uniquely and globally defined by the properties that \begin{itemize} \item $X_i\in \ker \omega_0$, and \item $\alpha_i(X_j)=\delta_{ij}$.\end{itemize}
$\ker\omega_0$ is integrable, so we have $[X_i,X_j]\in \ker\omega_0$, and moreover, for every $i,j,m=1,\ldots,k$, recalling that $\ker d\alpha_m\supset \ker\omega_0$, we have \[ 0=d\alpha_m(X_i,X_j)=\mathcal{L}_{X_i}(\alpha_m(X_j))-\mathcal{L}_{X_j}(\alpha_m(X_i))-\alpha_m([X_i,X_j]),\] and therefore $\alpha_m([X_i,X_j])=0$.  But for this to hold for every $m$, it must be that $[X_i,X_j]=0$.  Thus if $N$ is stable, it admits a $k$-tuple of commuting vector fields which form a basis for the tangent spaces of each of the ($k$-dimensional) leaves of the characteristic foliation $\mathcal{F}$, in particular implying that each leaf is parallelizable and that any closed leaf is a torus.  While this is not an overly stringent condition when $k=1$, in higher codimension it is rather stringent.

Accordingly, we would like to find a more flexible condition on the coisotropic submanifold $N$ which still forces $N$ to  manifest at least some of the interesting properties that were found by Ginzburg in \cite{G06}.  Recall that a foliation $\mathcal{F}$ on a (for convenience, closed) Riemannian manifold $(N,h)$ is called \emph{totally geodesic} if each geodesic in $N$ initially tangent to a leaf remains contained in the leaf; equivalently, with respect to the Levi-Civita connection $\nabla$ induced by $h$, for any vector field $X$ tangent to the foliation $\nabla_X X$ is also tangent to the foliation.  Foliations with this property are studied  in, \emph{e.g.}, \cite{JW}. As an example, if $\mathcal{F}$ is given by the fibers of a Riemannian submersion $\pi\co N\to B$, $\mathcal{F}$ is totally geodesic iff the structure group of the fiber bundle $\pi\co N\to B$ reduces to the isometry group of the fibers.  

We will show presently that  the stability of $N$  always implies that a metric $h$ can be found on $N$ with respect to which the characteristic foliation is totally geodesic, while these two notions are equivalent when the codimension of $N$ (and hence the rank of $\mathcal{F}$) is one.  In the codimension one case, this fact is closely related to the result of D. Sullivan \cite{Su} (see also \cite{Gluck}) stating that a vector field $V$ generates a geodesible foliation iff there is a codimension-one distribution $\xi$ transverse to $V$ which is preserved by the flow of $V$.  At the other extreme, where $N\subset M$ is Lagrangian, it is totally geodesic with respect to any metric for trivial reasons, while as mentioned in Example 2.2 (v) of \cite{G06} it can be stable only if it is a torus.

\begin{prop}\label{stabimpliesgeod} If $N\subset M$ is a closed stable coisotropic submanifold, then there is a metric $h$ on $N$ with respect to which the characteristic foliation on $N$ is totally geodesic.
\end{prop}
\begin{proof}
Let $\alpha_1,\ldots,\alpha_k\in \Omega^1(N)$ be as in the definition of stability, and as above let $X_1,\ldots,X_k$ be the vector fields tangent to $\mathcal{F}$ and defined by $\alpha_i(X_j)=\delta_{ij}$. Denote the time-$t$ maps of the flows of the vector fields $X_i$ by $\phi_{i}^{t}$.   

Now we have, for each $i$ and $j$ \[ \mathcal{L}_{X_i}\alpha_j=\iota_{X_i}d\alpha_j+d\iota_{X_i}\alpha_j=0\] since $\iota_{X_i}\alpha_j$ is the constant $\delta_{ij}$, while $X_i\in \ker\omega|_N\subset\ker d\alpha_j$.  Hence setting \[ \xi=\cap_{i=1}^{k}\ker\alpha_i,\] the codimension-$k$ distribution $\xi$ is transverse to the foliation $\mathcal{F}$ and is preserved by the flows $\phi_{i}^{t}$ of the various $X_i$.

$N$ can be covered by coordinate patches of the following kind.  Let $p\in N$, and, where $B\subset \mathbb{R}^{2n-2k}$ is a (sufficiently small) ball around the origin, let $\phi_p\co B\to N$ be an embedding which is transverse to the foliation $\mathcal{F}$, with $\phi_p(\vec{0})=p$.    Then for a sufficiently small ball $D$ around the origin in $\mathbb{R}^k$, the map $\Phi_p\co D\times B\to N$ defined by setting, for $(x_1,\ldots,x_k)\in D$ and $b\in B$,\[ \Phi_p(x_1,\ldots,x_k,b)=(\phi_1^{x_1}\circ\cdots\circ\phi_{k}^{x_k})(\phi_p(b))\] gives a coordinate patch around $p$.  Note that, as a result of the fact that the various $\phi_{i}^{x_i}$ mutually commute (since, as we've noted earlier, $[X_i,X_j]=0$), $(\Phi_p)_{*}\partial_{x_i}=X_i$ for each $i$.

Since $\xi\oplus span\{X_1,\ldots,X_n\}=TN$, it is then straightforward to find a metric $g'_p$ on $D\times B$ such that $g'_p(\partial_{x_i},\partial_{x_j})=\delta_{ij}$ and, for each $v\in (\Phi_p)_{*}^{-1}(\xi)$, we have $g'_p(\partial_{x_i},v)=0$ for all $i\in \{1,\ldots,k\}$.  Denote by $g_p$ the pushforward of $g'_p$ via $\Phi_p$.

This gives a metric $g_p$ on a neighborhood of $V_p=\Phi_p(D\times B)$ of an arbitrary point $p\in N$ with the properties that \begin{itemize} \item[(i)] $g_p(X_i,X_j)=\delta_{ij}$ and \item[(ii)] For each $i\in \{1,\ldots,k\}$ and each $v\in \xi$ we have $g_p(X_i,v)=0$.\end{itemize}  Since $N$ is compact, cover $N$ by such neighborhoods $V_{p_1},\ldots,V_{p_N}$ and let $\{\chi_i\}_{1\leq i\leq N}$ be a partition of unity subordinate to the $V_{p_i}$.  Then \[ h=\sum_{i=1}^{N}\chi_{i}g_{p_i}\] is a metric on all of $N$ with the same properties (i) and (ii) above.

We claim now that the foliation $\mathcal{F}$ is totally geodesic with respect to the metric $h$.  Since the tangent spaces to the leaves of $\mathcal{F}$ are spanned by the $X_i$ and are the orthogonal complements of $\xi$, for this it suffices to show that, if $V$ is a vector field tangent to $\xi$ and $W=\sum_{i=1}^{k}c_iX_i$ (where $c_i\in\mathbb{R}$), we have $h(\nabla_W W,V)=0$.

Now $h(W,V)$ vanishes identically, so we have \[ 0=W(h(W,V))=h(\nabla_W W,V)+h(W,\nabla_W V).\]  Meanwhile \[ \mathcal{L}_WV=[W,V]=\nabla_WV-\nabla_VW.\] But for each $i$ and $j$ we have \[ 0=\mathcal{L}_{X_i}(\alpha_j(W))=(\mathcal{L}_{X_i}\alpha_j)(W)+\alpha_j(\mathcal{L}_{X_i}V)=
\alpha_j(\mathcal{L}_{X_i}V),\] so since $W=\sum_ic_i X_i$ we have $\alpha_j(\mathcal{L}_WV)=0$.  Thus $[W,V]\in \xi$, and so \[ h([W,V],W)=h(\nabla_WV-\nabla_VW,W)=0.\]  Combining all this, we find \begin{align*} h(\nabla_WW,V)&=-h(W,\nabla_W V)=-h(W,\nabla_VW)\\&=-\frac{1}{2}V(h(W,W))=0\end{align*} since $h(W,W)=\sum_{i=1}^{k}c_{i}^{2}$ is constant.

This proves that if $W$ is any constant linear combination of the $X_i$ then $\nabla_WW$ is tangent to the foliation $\mathcal{F}$.  Since the $X_i$ span the tangent space to the foliation, the Leibniz rule then immediately proves the same statement when instead we have $W=\sum_{i=1}^{k}f_iX_i$ for some functions $f_i$ on $N$.  Since all vector fields $W$ tangent to the foliation have this latter form, this completes the proof that $\mathcal{F}$ is totally geodesic with respect to the metric $h$.

\end{proof}

\begin{prop} If $N$ is a closed oriented hypersurface in $M$ then $N$ admits a metric making its characteristic foliation totally geodesic if and only if $N$ is stable.
\end{prop}

\begin{proof} The backward implication has already been proven.  Conversely, suppose that $h$ is a metric on $N$ with the property that the characteristic foliation $\mathcal{F}$ of $N$ (which here has rank $1$) is totally geodesic with respect to $h$.  Since $M$ (being a symplectic manifold) is oriented, the normal bundle to $N$ inherits an orientation from the orientations of $N$ and $M$.  The orientation on the normal bundle to $N$ then induces one on the conormal bundle to $N$, and so the tangent bundle $\ker\omega|_N$ to the foliation $\mathcal{F}$ inherits an orientation via the natural $\omega$-induced isomorphism between $\ker\omega|_N$ and the conormal bundle to $N$.  We can therefore find a global section of the rank-1 bundle $\ker\omega|_N$, which we denote by $W$.  Rescaling if necessary, we may assume that $h(W,W)=1$.  Define a $1$-form $\alpha\in\Omega^1(N)$ by $\alpha(v)=h(v,W)$.  Let $\xi$ denote the orthogonal complement of $\ker\omega|_N$ with respect to the metric $h$; thus $\ker\omega|_N\oplus\xi=TN$, and in particular $\omega^{n-1}$ is nonvanishing on $\xi$.  If $V$ is a section of $\xi$, we have \[ d\alpha(V,W)=V(\alpha(W))-W(\alpha(V))-\alpha([V,W]).\]  The first term vanishes since $\alpha(W)=1$ identically, while the second vanishes since $\alpha(V)=0$ identically.  Now \begin{align*} \alpha([V,W])&=h(W,[V,W])=h(W,\nabla_VW)-h(W,\nabla_WV)\\&=  
\frac{1}{2}V(h(W,W))-\left(W(h(W,V))-h(\nabla_W W,V)\right)=0,\end{align*} for the first term vanishes since $h(W,W)=1$ identically; the second vanishes since $h(W,V)=0$ identically; and the third vanishes since, by virtue of $\mathcal{F}$ being totally geodesic, $\nabla_W W$ is tangent to $\mathcal{F}$ and therefore orthogonal to $\xi$.

Since $d\alpha(W,W)=0$ and $d\alpha(V,W)=0$ for all $V\in \xi$, it follows that $W\in \ker d\alpha$.  But $\ker\omega|_N=span\{W\}$, so $\ker \omega|_N\subset \ker d\alpha$.  Since $\omega|_N$ is nondegenerate on $\xi=\ker\alpha$, while the kernel of $\omega|_N$ is spanned by $W$ and $\alpha(W)=1$, we immediately obtain that $\alpha\wedge(\omega|_N)^{n-1}$ is  a volume form on $N$.  This proves the stability of the hypersurface $N$.
\end{proof}

We now turn to some issues relating to the symplectic geometry of coisotropic submanifolds satisfying these conditions.

Fix now a closed connected coisotropic submanifold $N\subset M$ (where $(M,\omega)$ is a closed symplectic manifold).  Choose a Riemannian metric $h$ on $N$; we make no assumptions on the behavior of $h$ with respect to the symplectic form $\omega$.  Write $\omega_0=\omega|_N$.

Define a distribution $E\subset TN$ by $E=\ker\omega_0$.  Thus the characteristic foliation $\mathcal{F}$ has $T\mathcal{F}=E$, and $E\to N$ is a vector bundle of rank equal to the codimension of $N$ in $M$.  
 
Define \[ \Pi_{h}\co TN\to E \] to be the orthogonal projection of $TN$ onto $E$ that is induced by the Riemannian metric $h$.

Let \[ \pi\co E^*\to N\] denote the dual vector bundle to $E$, and for $r>0$ let \[ E^{*}(r)=\{(x,p)\in E^*|x\in N,p\in(E_x)^*,|p|^{2}_{h}< r^2\} \] be the radius-$r$ disc bundle of $E^*$, where we measure the norm $|p|_h$ of $p$ in the obvious way using $h$.   

Define $\theta_{h}\in \Omega^1(E^*)$ by \[ (\theta_h)_{(x,p)}(v)=p(\Pi_h(\pi_*v)) \mbox{ for }v\in T_{(p,x)}E^*.\]

\begin{prop}[\cite{Marle},\cite{Gotay}]\begin{itemize} \item[(i)] The $2$-form \[ \omega_{E^*}\in \Omega^2(E^*)\mbox{ defined by } \omega_{E^*}=\pi^*\omega_0-d\theta_h \] restricts as a symplectic form to $E^{*}(R)$ for all sufficiently small $R>0$, and  the zero-section $N$ is a coisotropic submanifold with respect to this symplectic form.
\item[(ii)] For $R>0$ sufficiently small, there is an open neighborhood $U_R\subset M$ of $N$ and a symplectomorphism \[ \psi_R\co (E^*(R),\omega_{E^*})\to (U_R,\omega|_{U_R}),\] restricting as the identity on $N$.\end{itemize} 
\end{prop}
\begin{proof}(i) is established in Proposition 3.2 of \cite{Marle}.  Given (i), (ii) is the ``Local Uniqueness Theorem'' in \cite{Gotay} (which in turn follows fairly quickly from the Weinstein-Moser trick, and was also proven slightly later as a special case of Th\'eor\`eme 4.5 of \cite{Marle}).\end{proof}

The following was observed independently by B. Tonnelier \cite{Ton}.

\begin{lemma}\label{geodsame}  Assume that the characteristic foliation $\mathcal{F}$ is totally geodesic with respect to the metric $h$.  Then the Hamiltonian flow of the function $F=\frac{1}{2}|p|_{h}^{2}$ on $E^{*}(R)$ coincides with the geodesic flow of the metric $h$ on $T^*N$, restricted to the cotangent space $E^*$ of the foliation.
\end{lemma}

\begin{proof}  Denote by $\xi=E^{\perp_h}$ the orthogonal complement of $E$ with respect to the metric $h$.  Let \[ i\co E^*\to T^*N \] be the inclusion of $E^*$ into $T^*N$ obtained by extending each $p\in E^{*}_{x}$ to a linear functional on $T_xN$ via $p|_{\xi_x}=0$.  In particular \[ i(E^*)=\{(x,p)\in T^*N|\langle p,v\rangle=0\mbox{ for all }v\in\xi_x\}.\]

Define $\tilde{F}\co T^*N\to \mathbb{R}$ by $\tilde{F}(x,p)=\frac{1}{2}|p|_{h}^{2}$ for all $p\in T^*N$; of course we have $i^*\tilde{F}=F$.  Let $\lambda\in \Omega^1(T^*N)$ be the canonical $1$-form (so where $\pi'\co T^*N\to N$ is the bundle projection we have $\lambda_{(x,p)}(v)=p(\pi'_{*}v)$), and endow $T^*N$ with its standard symplectic form $-d\lambda$. 

Note that, for $(x,p)\in E^*$ (so $p\in E_{x}^{*}$), we have $i(x,p)=(x,p\circ\Pi_h)\in T^*N$, where as before $\Pi_h\co TN\to E$ denotes the orthogonal projection with respect to $h$.   Hence, if $v\in T_{(x,p)}E^*$, \begin{align*} 
(i^*\lambda)_{(x,p)}(v)&=\lambda_{i(x,p)}(i_*v)=\lambda_{(x,p\circ\Pi_h)}(i_*v)\\&=(p\circ\Pi_h) (\pi'_*(i_*v))=p\circ\Pi_h (\pi_*v)=p(\Pi_h (\pi_*v))=(\theta_h)_{(x,p)}(v);\end{align*} thus, \begin{equation}\label{lambdatheta}
i^*\lambda=\theta_h.\end{equation}

 As is well-known (and can be seen via an easy calculation in geodesic coordinates), the Hamiltonian vector field $X_{\tilde{F}}$ of $\tilde{F}$ with respect to the standard symplectic form $-d\lambda$ induces the geodesic flow of $h$ on $T^*N$.  
In other words, a curve $\gamma\co [0,T]\to T^*N$ given by $\gamma(t)=(x(t),p(t))$ with $p(t)\in T_{x(t)}^{*}N$ is an integral curve for $X_{\tilde{F}}$ if and only if $t\mapsto x(t)$ is a geodesic and $p(t)$ is  dual to $\dot{x}(t)$ with respect to the metric $h$.

We claim that the fact that $\mathcal{F}$ (which, we recall, has $T\mathcal{F}=E$) is totally geodesic implies that $X_{\tilde{F}}$ is tangent to $i(E^{*})$.  Indeed, if $t\mapsto \gamma(t)=(x(t),p(t))$ is an integral curve of $X_{\tilde{F}}$ such that $\gamma(0)\in i(E^*)$, and if $V$ is a local section of $\xi$, then $t\mapsto x(t)$ gives a geodesic which remains in the same leaf of $\mathcal{F}$, and so \[ \langle p(t),V(x(t))\rangle=h(\dot{x}(t),V(x(t)))=0\mbox{ for all $t$},\] since $\dot{x}(t)\in E_{x(t)}$ for all $t$ while $V(x(t))\in (E_{x(t)})^{\perp_h}$.  Thus (since $(x,p)\in i(E^*)$ iff $p$ annihilates $\xi$), any integral curve of $X_{\tilde{F}}$ initially contained in $i(E^*)$ remains in $i(E^*)$, confirming that $X_{\tilde{F}}$ is tangent to  $i(E^{*})$.  

Thus there is a vector field $X$ on $E^*$ such that $i_*X=X_{\tilde{F}}$, and the content of the Lemma is that, restricting to the region $E^{*}(R)$ where $\omega_{E^*}$ is symplectic, $X$ is the Hamiltonian vector field of $F$.

Where $X$ is the vector field on $E^*$ characterized by $i_*X=X_{\tilde{F}}$, we have for $v\in TE^*$, \[
(\iota_{X}\pi^*\omega_0)(v)=((\pi'\circ i)^*\omega_0)(X,v)=((\pi')^*\omega_0)(i_*X,i_*v)=\omega_0(\pi'_*X_{\tilde{F}},\pi_*v).\]  But since an integral curve of $X_{\tilde{F}}$ through a point of $E^*$ projects to a geodesic initially (indeed always) perpendicular to $\xi$ and so tangent to $E$, we have $\pi'_*X_{\tilde{F}}\in E=\ker(\omega_0)$.  Thus \[ \iota_X\pi^*\omega_0=0.\]

So for $v\in TE^*$, we have \begin{align*}
\omega_{E^*}(X,v)&=(\pi^*\omega_0-d\theta_h)(X,v)=-d\theta_h(X,v)=-(d(i^*\lambda))(X,v)\\&=-d\lambda(i_*X,i_*v)=-d\lambda(X_{\tilde{F}},i_*v)=d\tilde{F}(i_*v)\\&=d(i^*\tilde{F})(v)=dF(v),\end{align*}
where the third equality uses (\ref{lambdatheta}).  

Thus, in the region $E^{*}(R)$ where $\omega_{E^*}$ is symplectic, $X$ is the Hamiltonian vector field of $F\co E^{*}(R)\to\mathbb{R}$, completing the proof of the lemma. 
\end{proof}

\begin{theorem} \label{displaceable} Let $N$ be a closed, displaceable, coisotropic submanifold of the closed symplectic manifold $(M,\omega)$.  If the characteristic foliation of $N$ is made totally geodesic by the metric $h$ on $N$, then the Riemannian manifold $(N,h)$ has a closed geodesic which is tangent to $\ker(\omega|_N)$ and which is contractible in $M$.
\end{theorem}

\begin{proof}  For any sufficiently small $R>0$ we may 
 symplectically identify a neighborhood $U_R$ of $N$ in $M$ with $(E^{*}(R),\omega_{E^*})$.  
 Since $N$ is displaceable (say $\phi_{H}^{1}(N)\cap N=\varnothing$), there is $R>0$ such that $\phi_{H}^{1}(\overline{U_R})\cap \overline{U_R}=\varnothing$.  By the energy-capacity inequality (\cite{U3}, Theorem 1.1), it follows that the $\pi_1$-sensitive Hofer-Zehnder capacity of $U_R$ is at most $\|H\|$, which is to say, if $K\co M\to\mathbb{R}$ is a  smooth function which is supported in $U_R$ and has $\max K>\|H\|$, then $X_K$ has a contractible periodic orbit of period at most one.
 
Choose a monotone decreasing smooth  function $f\co [0,\infty)\to [0,2\|H\|]$ such that $f(0)=2\|H\|$ and $f(s)=0$ for $s\geq R^2/2$.
Under the identification of $U_R$ with $E^{*}(R)$, where $F(x,p)=\frac{1}{2}|p|_{h}^{2}$ as in Lemma \ref{geodsame}, define $K_0\co U_R\to\mathbb{R}$ by  $K_0=f\circ F $, and then define $K\co M\to\mathbb{R}$ by $K|_{U_R}=K_0$ and $K|_{M\setminus U_R}=0$.  $X_K$ then has a nonconstant contractible periodic orbit of period at most one, which is necessarily contained in $U_R$.   
In $U_R$ we have $X_K=f'(F)X_F$, so it follows that $X_F$ has a nonconstant contractible periodic orbit $\gamma$ in $U_R\cong E^{*}(R)$.  Where $\pi\co E^{*}(R)\to N$ is the projection, by Lemma \ref{geodsame} $\pi\circ \gamma$ is then a geodesic in $N$ tangent to a leaf of the characteristic foliation, which is homotopic to $\gamma$ within $U_R$ and hence is contractible in $M$.

\end{proof}

\begin{remark}Using the stable energy-capacity inequality (\cite{Schlenk} Theorem 1.1, \cite{U3} Corollary 1.3), we see that Theorem \ref{displaceable} applies equally well if $N$ is stably displaceable (\emph{i.e.}, $N\times S^1$ is displaceable in $M\times T^*S^1$).

\end{remark}

\section{Consequences of the boundary depth for displacement energy}\label{last}

Throughout this section we assume that the image of the homomorphism \[ \langle\omega|_N,\cdot\rangle\co \pi_2(N)\to\mathbb{Z}\] is discrete; let $\lambda_0$ denote the positive generator of this image if $\langle\omega|_N,\cdot\rangle$ is nontrivial, and $\lambda_0=\infty$ if $\langle\omega|_N,\cdot\rangle$ is trivial.

Assume, further, that $R>0$ is such that there is a symplectic embedding \[ \psi_R\co (E^{*}(R),\omega_{E^*})\hookrightarrow (M,\omega).\] Let \[ V=\psi_R(E^{*}(R))\mbox{ and } U=\psi_R(E^{*}(R/2)).\]
Fix an almost complex structure $J_0$ on $\overline{E^{*}(R)}$ compatible with $\omega_{E^*}$.

This determines a constant $c_{UV}$, depending only on the data $(N,\omega|_N,R,J_0)$, as in Proposition \ref{energypartial}.
In the notation of Proposition \ref{flatlemma}, we have $\lambda_0(V)=\lambda_0$.

Let $\ep>0$ and let $H\co (\mathbb{R}/\mathbb{Z})\times M\to\mathbb{R}$ be a Hamiltonian with the property that \[ 
\|H\|<\min\{c_{UV},(3+\ep)^{-1}\lambda_0\} \] and \[ \phi_{H}^{1}(N)\cap N=\varnothing.\]  There is then $r>0$ such that \[ \phi_{H}^{1}(\psi_R(\overline{E^{*}(r)}))\cap \psi_R(\overline{E^{*}(r)})=\varnothing; \] without loss of generality we may assume that $r<R/2$.  

Choose any smooth, monotone  function $f\co [0,\infty)\to [-(1+\ep)\|H\|,0]$ such that $0$ is the unique global minimum of $f$, with $f(0)=-(1+\ep)\|H\|$, such that $f'$ vanishes to infinite order at $s=0$, and such that $f(s)=0$ for $s\geq r$.

Define $K\co(\mathbb{R}/\mathbb{Z})\times M\to \mathbb{R}$ by setting $K(t,m)=0$ if $m\notin V$ and, under the identification of $V$ with $E^*(R)$ via $\psi_R$, setting $K(t,(x,p))=f(|p|_{h})$ for $(x,p)\in E^*(R)\cong V$.  Our choice of $f$ ensures that, for any $q\in N$, $K$ has a flat autonomous minimum at $q$.  (In the notation of Definition \ref{flat}, we have $S=N$, and $G$ can be taken to be any open set containing $N$.)  We have \[ \|H\|<-K(0,q)=(1+\ep)\|H\|,\] and \[ 2\|H\|+\|K\|=(3+\ep)\|H\|<\lambda_0=\lambda_0(V).\]  

Hence Proposition \ref{flatlemma} applies to $K$ to yield:

\begin{theorem}\label{finalorbit} Under the above hypotheses, for any almost complex structure $J_0$ on $V\cong E^*(R)$ there is a solution $u\co\mathbb{R}\times\mathbb{R}/\mathbb{Z}\to \overline{E^{*}(R)}$ to \begin{equation}\label{finaleqn} \frac{\partial u}{\partial s}+J_0\left(\frac{\partial u}{\partial t}-X_{K}\right)=0, \end{equation} such that $E(u)\leq 2\|H\|$, and, for some $q\in N$ and some nonconstant periodic orbit $\gamma$ of $X_{K}$, $u(s,\cdot)$ is partially asymptotic to the constant orbit $\gamma_q$ at $q$ as $s\to\infty$, and $u(s,\cdot)$ is partially asymptotic to $\gamma$ as $s\to -\infty$.
\end{theorem}

We are now in position to prove our main results on the displacement energy of coisotropic submanifolds.

\begin{theorem}\label{geodmain}  Suppose that $N$ is a closed coisotropic submanifold of the closed symplectic manifold $(M,\omega)$, that there is a Riemannian metric $h$ on $N$ with respect to which the characteristic foliation of $N$ is totally geodesic, and that $\{\int_{S^2}v^*\omega|v\co S^2\to N\}$ is discrete.  There is $c>0$, depending only on a tubular neighborhood of $N$ in $M$, with the following property.  If $h$ carries no closed geodesics which are contractible in $N$ and are contained in a leaf of the characteristic foliation, then the displacement energy of $N$ is at least $c$.
\end{theorem}

\begin{proof}  Take $c$ equal to the minimum of the constants $c_{UV}$ and $\lambda_0/3$ from the start of this section.  If $H$ is a Hamiltonian displacing $N$ with $\|H\|<c$, then  for some $\ep>0$ we have $\|H\|<\min\{c_{UV},(3+\ep)^{-1}\lambda_0\}$, so we are in the situation above.  We thus have an autonomous Hamiltonian $K\co \overline{E^{*}(R)}\to\mathbb{R}$ having the form $K=\alpha\circ F$ for a certain smooth function $\alpha\co \mathbb{R}\to\mathbb{R}$, such that $X_{K}$ has a nonconstant $1$-periodic orbit which is contractible in $\overline{E^*(R)}$.  But such an orbit is a reparametrization of a nonconstant contractible periodic orbit $\gamma$ of $X_F$, and by Lemma \ref{geodsame} $X_F$ generates the leafwise geodesic flow.  Where $\pi\co E^*(R)\to N$ is the bundle projection, $\pi\circ \gamma$ is a geodesic which is tangent to the leaves of the foliation, and is homotopic to $\gamma$ within $E^{*}(R)$ and hence contractible in $E^*(R)$ and so also in $N$.  This contradicts the hypothesis of the theorem.
\end{proof}

\begin{cor}  If $N$ is a closed coisotropic submanifold of $(M,\omega)$ and if $N$ admits a metric $h$ of nonpositive curvature with respect to which the characteristic foliation on $N$ is totally geodesic, then $N$ has positive displacement energy.
\end{cor}
\begin{proof}  The fact that $h$ has nonpositive curvature implies that none of its closed geodesics are contractible, and further (by the Cartan--Hadamard theorem) that $N$ has contractible universal cover and so $\pi_2(N)=0$.   Hence we may apply Theorem \ref{geodmain}.
\end{proof}

\begin{theorem}\label{stablemain} If $N$ is a stable coisotropic submanifold of the closed symplectic manifold $(M,\omega)$, and if
$\{\int_{S^2}v^*\omega|v\co S^2\to N\}$ is discrete, then there is a constant $c>0$ depending only on a tubular neighborhood of
$N$ in $M$ with the property that the displacement energy of $N$ in $M$ is at least $c$.
\end{theorem}

\begin{proof} The argument is just as in Section 6 of \cite{G06}.  Assume that $\phi_{H}^{1}(N)\cap N=\varnothing$ and that $\|H\|< \min\{c_{UV},(3+\ep)^{-1}\lambda_0\}$, so that Theorem \ref{finalorbit} applies  (of course, if this is not the case for any $H$, the statement of this theorem follows vacuously).   If $\alpha_1,\ldots,\alpha_k\in \Omega^1(N)$ are as in the definition of stability (so $\alpha_1\wedge\cdots\wedge\alpha_k\wedge(\omega|_N)^{n-k}$ is a volume form on $N$ and $\ker(\omega|_N)\subset\ker\omega$), just as in Equation 6.3 on p. 150 of \cite{G06}\footnote{Note that our convention for the sign of a Hamiltonian vector field is opposite to that of \cite{G06}} one shows that, for $i=1,\ldots,k$ and some constants $c_1,\ldots,c_k$, the $u$ and $\gamma$ produced by   Theorem \ref{finalorbit} obey \[ E(u)\geq c_{i}^{-1}\left|\int_{\pi\circ\gamma}\alpha_i\right|.\]  With respect to the metric produced in Proposition \ref{stabimpliesgeod}, $\pi\circ\gamma$ is a closed geodesic, with length $\sum_{i=1}^{k}x_i\int_{\pi\circ\gamma}\alpha_i$ for some $x_1,\ldots,x_k$ with $\sum_{i=1}^{k}x_{i}^{2}=1$.  ($x_i$ are determined by the condition that the  unit tangent vector field  to the geodesic is $\sum x_iX_i$, where as before $\alpha_i(X_j)=\delta_{ij}$).  Thus \[ Length(\pi\circ\gamma)\leq (\sum_{i=1}^{k}|x_i|c_i)E(u)\leq (\sum_i c_{i}^{2})^{1/2}E(u)\leq 2(\sum_i c_{i}^{2})^{1/2}\|H\| \] (recalling that $E(u)\leq 2\|H\|$).  So where $\Lambda$ is the minimal length of a closed geodesic in $N$ we obtain the lower bound $\|H\|\geq (4\sum_i c_{i}^{2})^{-1/2}\Lambda$.
\end{proof}

\begin{cor} If $N$ is a coisotropic submanifold of contact type in the closed symplectic manifold $(M,\omega)$ then $N$ has positive displacement energy.
\end{cor}

\begin{proof}
The contact type condition ensures that $\omega|_N$ is exact and so vanishes on $\pi_2(N)$, so we can apply the previous theorem.
\end{proof}

\begin{cor}\label{stein} Theorems \ref{geodmain} and \ref{stablemain} continue to apply if the closed symplectic manifold $(M,\omega)$ is replaced by a Stein manifold $(S,J,dd^c\psi)$.  In particular, any closed stable coisotropic submanifold of a Stein manifold has positive displacement energy.
\end{cor}

\begin{proof} Suppose that $H\co (\mathbb{R}/\mathbb{Z})\times S\to\mathbb{R}$ is a (compactly supported) Hamiltonian on $S$ with $\phi_{H}^{1}(N)\cap N=\varnothing$, and suppose $U\subset S$ is an open set with compact closure such that $(\mathbb{R}/\mathbb{Z})\times U$ contains the support of $H$.  Theorem 3.2 of \cite{LiMa} ensures that there is a symplectic embedding of $ U\hookrightarrow M$ into a closed symplectic manifold $M$.  But then applying our previous results to $M$ shows that, for some constant $c$ depending only on a tubular neighborhood of $N$ in $S$ (and in particular not otherwise depending on $M$ or on the size of the support of $H$), we have $\|H\|\geq c$.  So $N$ has displacement energy at least $c$.

For the last sentence, simply note that the symplectic form on a Stein manifold is exact, so if $N$ is a stable coisotropic submanifold of the Stein manifold $(S,\omega)$ then  $\omega|_N$ will be exact and so will vanish on $\pi_2(N)$.

\end{proof}

\end{document}